\documentclass[a4paper, 12pt]{article}
\usepackage{mystyle}

\begin{document}
	\onehalfspacing
	
	\title{Equivariant $KK$-theory for non-Hausdorff groupoids} 
	\author{Lachlan E. MacDonald}

	\date{December 2019}
	
	\maketitle
	
\begin{abstract}
	We give a detailed and unified survey of equivariant $KK$-theory over locally compact, second countable, locally Hausdorff groupoids.  We indicate precisely how the ``classical" proofs relating to the Kasparov product can be used almost word-for-word in this setting, and give proofs for several results which do not currently appear in the literature.
\end{abstract}

This article is intended as a detailed survey of the theory of groupoid-equivariant $KK$-theory, in the setting of locally compact, second countable, locally Hausdorff groupoids.  Following a substantial period of innovation and development \cite{KasparovTech,cs,sk1,higsonkastech,KasparovEqvar} in the non-equivariant and group-equivariant cases, groupoid-equivariant $KK$-theory was first treated, for Hausdorff groupoids, by Le Gall in his PhD thesis \cite{legall}.  Since many examples of groupoids are only locally Hausdorff (for instance, the holonomy groupoids of foliated manifolds \cite{folops}),  it is desirable to have a groupoid-equivariant $KK$-theory for \emph{non-Hausdorff} groupoids as well.

Despite the demand arising from applications to foliation theory, the literature on equivariant $KK$-theory over non-Hausdorff groupoids has, up until this point, been rather sparse.  The non-Hausdorff theory makes its first appearance in \cite{tu}, in which, for details on the inner workings of the theory, the reader is referred to the expositions given for Hausdorff groupoids in \cite{legall2,twistedKstacks}.  Since then the theory has been very rapidly expounded upon, in an extremely general context that covers not only groupoids but Hopf algebras, in \cite{iakovos} for applications to singular foliations.  Finally, the theory has found use in \cite{macr1} in the study of the Godbillon-Vey invariant of regular foliated manifolds.

The common denominator in all of the treatments \cite{tu,iakovos,macr1} of equivariant $KK$-theory over non-Hausdorff groupoids is that each has been focussed on a particular application of equivariant $KK$-theory, rather than on a detailed development of the theory itself.  As a consequence, most details (some of which are somewhat nontrivial) are skipped over, and it is difficult to glean a precise understanding of how the non-Hausdorffness of the underlying groupoid affects the ``classical" definitions, statements of results, and their proofs. The goal of the present paper is to rectify this state of the literature, by providing a unified and detailed description of the theory.  In particular, we include all the necessary definitions (of which there are many), and include detailed proofs for several results which do not yet appear in the literature.  We will see, as remarked in \cite{iakovos}, that much of the theory can be proved in exactly the manner of the ``classical" theory, provided one makes some conceptual substitutions.

Let us briefly describe the content of the paper.  The first three sections consist of the relevant definitions for groupoids, upper-semicontinuous bundles and actions of groupoids on algebras and Hilbert modules.  Note that we do \emph{not} go to the generality of Fell bundles, and we refer the reader to \cite{twistedKstacks} and \cite{tu} for an indication of how the theory works at this level of generality.  The bundle-theoretic approach we take to equivariant $KK$-theory is heavily influenced by \cite{felldoran,williamscrossed,muhwil}, which are referenced throughout the paper.  The following two sections consist of definitions and results regarding equivariant $KK$-theory, its functoriality, and the Kasparov product.  We have made a conscious choice not to derive our proofs from those given by Le Gall \cite{legall,legall2}, whose thesis and paper on the topic can be relatively difficult to find.  Instead we closely follow the ``classical" theory as detailed in \cite{KasparovTech,sk1,higsonkastech,KasparovEqvar}, inspired by the very modern approach of \cite{iakovos}.  The final two sections are concerned with crossed products, both full and reduced, and the associated descent map on equivariant $KK$-theory, for which we draw greatly on \cite{koshskand,muhwil}.  Throughout the exposition, we include a description of the unbounded picture \cite{baajjulg1} for future applications where a need to compute explicit representatives of Kasparov products \cite{kucer,LRV, mes14, kl13, bmv, mr15} and index formulae \cite{CM,higsonlocind, cprs1, cprs2, cprs3, cprs4, cgrs1, cgrs2} may arise.  The results regarding the unbounded picture have already appeared, in the context of Lie groupoids, in \cite{macr1}, which itself draws on the theory outlined in \cite{pierrot} in the Hausdorff setting.

\subsection{Acknowledgements}

I thank the Australian Federal Government for financing my doctoral research, of which the present survey forms a part, via a Research Training Program scholarship.  Special thanks go to A. Rennie, for his consistent support over the past several years, and to A. Carey, for encouraging me to write this article.

\section{Groupoids}

We begin with the basic definitions that we require regarding groupoids and their actions.  

\begin{defn}\label{groupoid}
	A \textbf{groupoid} $(\GG^{(0)},\GG,r,s,m,i)$ consists of a set $\GG$, the \textbf{total space}, with a distinguished subset $\GG^{(0)}$, the \textbf{unit space}, maps $r,s:\GG\rightarrow \GG^{(0)}$ called the \textbf{range} and \textbf{source} respectively, a \textbf{multiplication map} $m:\GG\times_{s,r}\GG\rightarrow \GG$ and an \textbf{inversion map} $i:\GG\rightarrow \GG$ such that
	\begin{enumerate}
		\item $i$ satisfies $i^{2} = \id$,
		\item $m$ is associative,
		\item $r\circ i = s$.
		\item $m(u,x) = u = m(y,u)$ for all $u\in \GG$ with $r(u) = y$ and $s(u) = x$.
	\end{enumerate}
We say that $\GG$ is \textbf{locally compact, second-countable, and locally Hausdorff} if its total space admits a locally compact, second-countable and locally Hausdorff topology for which all of the maps $r,s,m,i$ are continuous, and for which $r$ and $s$ are open.  We will always assume in addition that the unit space $\GG^{(0)}$ is Hausdorff.  A \textbf{homomorphism} $\varphi:\HH\rightarrow\GG$ of locally compact, second-countable, locally Hausdorff groupoids is a continuous map of their total spaces which is compatible with the multiplication and inversion maps on each.
\end{defn}

We will usually denote a groupoid $(\GG^{(0)},\GG,r,s,m,i)$ by simply $\GG$, $m(u,v)$ by simply $uv$, and $i(u)$ by $u^{-1}$.

\begin{defn}\label{gspace}
	Let $X$ be a topological space.  A \textbf{left action} of $\GG$ on $X$ consists of a continuous map $p:X\rightarrow \GG^{(0)}$ called the \textbf{anchor map}, together with a continuous map $\alpha:\GG\times_{s,p}X\rightarrow X$, denoted $\alpha(u,x)\mapsto \alpha_{u} (x)$, such that
	\begin{enumerate}
		\item $p(\alpha_{u}(x)) = r(u)$ for all $(u,x)\in \GG\times_{s,p}X$,
		\item $\alpha_{uv}(x) = \alpha_{u}(\alpha_{v}(x))$ for all $(v,x)\in \GG\times_{s,p}X$ and $(u,v)\in \GG^{(2)}$,
		\item $\alpha_{p(x)}(x) = x$ for all $x\in X$.
	\end{enumerate}
\end{defn}

\section{Upper-semicontinuous bundles}

We give a brief study of upper-semicontinuous bundles of spaces, algebras and modules.  We refer primarily to the delightful books \cite{felldoran} and \cite{williamscrossed}.  Concerning notation, for any map $f:Y\rightarrow X$ of sets, we will for each $x\in X$ denote by $Y_{x}$ the fibre $f^{-1}\{x\}$ over $x$.  We will moreover let $\mathbb{K}$ denote either $\CB$ or $\RB$.  Note that, just as in \cite{KasparovEqvar}, all of what follows in this article is valid for both complex and real $C^{*}$-algebras and Hilbert modules.  The notation $C_{0}(X)$, defined for locally compact spaces $X$, is used to mean the continuous $\mathbb{K}$-valued functions vanishing at infinity on $X$.

\begin{defn}\label{uscbb}
	Suppose that $X$ is a topological space.  By an \textbf{upper-semicontinuous Banach-bundle} over $X$, we mean a topological space $\AF$ together with an open surjective map $p_{\AF}:\AF\rightarrow X$ and a $\mathbb{K}$-Banach space structure on each fibre $\AF_{x}$ such that
	\begin{enumerate}
		\item the map $a\mapsto\|a\|$ is upper semicontinuous from $\AF$ to $\mathbb{R}$ (that is, for all $\epsilon>0$, $\{a\in\AF:\|a\|\geq\epsilon\}$ is closed),
		\item the map $+:\AF\times_{p_{\AF},p_{\AF}}\AF\rightarrow \AF$ given by $(a,b)\mapsto a+b$ is continuous,
		\item the map $\mathbb{K}\times\AF\rightarrow\AF$ given by $(\lambda,a)\mapsto\lambda a$ is continuous,
		\item if $\{a_{i}\}$ is a net in $\AF$ such that $p_{\AF}(a_{i})\rightarrow x$ and $\|a_{i}\|\rightarrow 0$, then $a_{i}\rightarrow 0_{x}$, where $0_{x}$ is the zero element in $\AF_{x}$.
	\end{enumerate}
	A \textbf{section} of $p_{\AF}$ is a function $\beta:X\rightarrow\AF$ such that $p_{\AF}\circ\beta = \id_{X}$.  The space of continuous sections of $p_{\AF}$ is denoted $\Gamma(X;\AF)$, the space of continuous sections that are bounded in the norm topology denoted $\Gamma_{b}(X;\AF)$, and the space of continuous sections vanishing at infinity in the norm topology on the fibres is denoted $\Gamma_{0}(X;\AF)$.
	
	If $Y$ is a topological space and $\phi:Y\rightarrow X$ is a continuous map, we denote by $\phi^{*}\AF$ the \textbf{pullback} of $\AF$ by $\phi$ to an upper-semicontinuous bundle over $Y$, whose fibre over $y\in Y$ is $\AF_{\phi(y)}$, and for any continuous section $a\in\Gamma(X;\AF)$ we denote by $\phi^{*}a$ its pullback to a section of $\phi^{*}\AF$ defined for $y\in Y$ by the formula $(\phi^{*}a)(y):=a(\phi(y))$.
\end{defn}

Some important special cases of upper-semicontinuous Banach bundes occur when the fibres are $C^{*}$-algebras, or Hilbert modules thereover.  While upper-semicontinuous bundles of $C^{*}$-algebras are by now quite well-studied, their Hilbert module analogues have not yet made an appearance in the literature.

\begin{defn}
	An upper-semicontinuous Banach bundle $p_{\AF}:\AF\rightarrow X$ is said to be an \textbf{upper-semicontinuous $C^{*}$-bundle} if each $\AF_{x}$ is a $C^{*}$-algebra, for which the multiplication $\cdot:\AF\times_{p_{\AF},p_{\AF}}\AF\rightarrow\AF$ is continuous.  We  say that $\AF$ is \textbf{$\ZB_{2}$-graded} if it admits a direct sum decomposition $\AF = \AF^{(0)}\oplus\AF^{(1)}$ such that for each $x\in X$, $\AF_{x} = \AF^{(0)}_{x}\oplus\AF^{(1)}_{x}$ is a $\ZB_{2}$-graded $C^{*}$-algebra (cf. \cite[Section 2.1]{KasparovTech}).
	
	If $p_{\AF}:\AF\rightarrow X$ is an upper-semicontinuous $C^{*}$-bundle, an \textbf{upper-semicontinuous Hilbert $\AF$-module bundle} is an upper semicontinuous Banach bundle $p_{\EF}:\EF\rightarrow X$ such that each $\EF_{x}$ is a Hilbert $\AF_{x}$-module, with $\AF_{x}$-valued inner product $\langle\cdot,\cdot\rangle_{x}$, and for which the maps $\langle\cdot,\cdot\rangle:\EF\times_{p_{\EF},p_{\EF}}\EF\rightarrow\AF$ and $\EF\times_{p_{\EF},p_{\AF}}\AF\rightarrow\EF$ induced by the inner product and right action respectively are continuous.  We say that such a bundle is \textbf{$\ZB_{2}$-graded} if it admits a direct sum decomposition $\EF = \EF^{(0)}\oplus\EF^{(1)}$ such that for each $x\in X$, $\EF_{x} = \EF^{(0)}_{x}\oplus\EF^{(1)}_{x}$ is a $\ZB_{2}$-graded Hilbert $\AF_{x}$-module (cf. \cite[Section 2.2]{KasparovTech}).
\end{defn}

\begin{rmk}\normalfont
	For the remainder of this paper, we will \emph{always} assume our algebra and Hilbert module bundles to be $\ZB_{2}$-graded.  In particular, all $C^{*}$-algebras and Hilbert modules are assumed to be $\ZB_{2}$-graded, and all tensor products and commutators are assumed to be $\ZB_{2}$-graded also (cf. \cite[Section 2]{KasparovTech}).
\end{rmk}

\begin{rmk}\normalfont
	For an upper-semicontinuous Banach bundle $p_{\BF}:\BF\rightarrow X$ over an arbitrary topological space $X$, there is no guarantee of a wealth of continuous sections of $p_{\BF}$.  When $X$ is locally compact and Hausdorff, however, by results of Hofmann \cite{hofmann} (see also the discussion in \cite[p. 16]{muhwil}), for any $x\in X$ and $b\in\BF_{x}$, one is guaranteed a section $\sigma$ for which $\sigma(x) = b$.  For the rest of this section we will \emph{always let $X$ denote a locally compact, Hausdorff space}.
\end{rmk}

Observe that if $p_{\AF}:\AF\rightarrow X$ is an upper-semicontinuous $C^{*}$-bundle, then the continuous sections vanishing at infinity $\Gamma_{0}(X;\AF)$ themselves form a $C^{*}$-algebra under pointwise operations and the supremum norm, and inherit a natural $\ZB_{2}$-grading from the fibres.  Letting $\MM(\AF)$ denote the associated upper-semicontinuous $C^{*}$-bundle whose fibre over $x\in X$ is equal to $\MM(\AF_{x})$, we can naturally identify $\MM(\Gamma_{0}(X;\AF))$ with the section algebra $\Gamma_{b}(X;\MM(\AF))$.  Note moreover that the $C^{*}$-algebra $\Gamma_{0}(X;\AF)$ associated to any upper-semicontinuous $C^{*}$-bundle $p_{\AF}:\AF\rightarrow X$ admits a nondegenerate homomorphism $C_{0}(X)\rightarrow Z\MM(\Gamma_{0}(X;\AF))$ into the center of $\MM(A)$, defined by pointwise scalar multiplication.  That is, $\Gamma_{0}(X;\AF)$ is a \emph{$C_{0}(X)$-algebra} in the following sense.

\begin{defn}
	A $C^{*}$-algebra $A$ is called a \textbf{$C_{0}(X)$-algebra} if it admits a nondegenerate homomorphism $C_{0}(X)\rightarrow Z\MM(A)$.  A $C^{*}$-homomorphism of $C_{0}(X)$-algebras is said to be a \textbf{$C_{0}(X)$-homomorphism} if it is in addition a homomorphism of $C_{0}(X)$-modules.
\end{defn}

Given any $C_{0}(X)$-algebra $A$, for each $x\in X$ one defines a $C^{*}$-algebra (and left $A$-module) $\AF_{x}:=A/(I_{x}\cdot A)$, where $I_{x}$ is the kernel of the evaluation functional $f\mapsto f(x)$ on $C_{0}(X)$.  The resulting space $\AF:=\bigsqcup_{x\in X}\AF_{x}$ is topologised by \cite[Theorem II.13.18]{felldoran} and, with the obvious projection $p_{\AF}:\AF\rightarrow X$, becomes an upper-semicontinuous $C^{*}$-bundle \cite[Theorem C.25]{williamscrossed}.  The $C_{0}(X)$-algebra $A$ is then  $C_{0}(X)$-isomorphic to $\Gamma_{0}(X;\AF)$ \cite[Theorem C.27]{williamscrossed}.  Thus $C_{0}(X)$-algebras are the same thing as sections of upper-semicontinuous $C^{*}$-bundles, and we will without further comment assume $A$ to be identified with its associated section algebra $\Gamma_{0}(X;\AF)$ for the remainder of the document.  Notice that any $C_{0}(X)$-homomorphism $\phi:A\rightarrow B$ of $C_{0}(X)$-algebras descends for each $x\in X$ to a $C^{*}$-homomorphism $\phi_{x}:\AF_{x}\rightarrow\BF_{x}$ of the fibres.

In a similar fashion, if $A = \Gamma_{0}(X;\AF)$ is a $C_{0}(X)$-algebra and $\EE$ is a Hilbert $A$-module, then for each $x\in X$ one defines the Hilbert $\AF_{x}$-module $\EF_{x}$ to be the balanced tensor product $E\hotimes_{A}\AF_{x}$.  The space $\EF:=\bigsqcup_{x\in X}\EF_{x}$ may also be topologised by means of \cite[Theorem II.13.18]{felldoran} to obtain an upper-semicontinuous Hilbert $\AF$-module bundle $p_{\EF}:\EF\rightarrow X$, where one proves the continuity of the $\AF$-valued inner products and of the right action of $\AF$ on $\EF$ using the same estimates as in the verification of Axiom A3 in \cite[Theorem C.25]{williamscrossed}.  The continuous sections vanishing at infinity $\Gamma_{0}(X;\EF)$ of this bundle are equipped with pointwise operations to furnish a Hilbert $\Gamma_{0}(X;\AF)$-module, to which $\EE$ is canonically isomorphic as a Hilbert $A$-module.  We will without further comment assume that $\EE$ is identified with its associated section space $\Gamma_{0}(X;\EF)$.

Given such a Hilbert module $\EE = \Gamma_{0}(X;\EF)$, we have associated upper-semicontinuous bundles of $C^{*}$-algebras $\KK(\EF)$ and $\LL(\EF)$, whose fibres over $x\in X$ are $\KK(\EF_{x})$ and $\LL(\EF_{x})$ respectively.  By the identification $\EE = \Gamma_{0}(X;\EF)$ we then have $\KK(\EE) = \Gamma_{0}(X;\KK(\EF))$ and $\LL(\EE) = \Gamma_{b}(X;\LL(\EF))$.  The natural $C_{0}(X)$-module structures on these spaces defined by pointwise multiplication then agree with \cite[Definition 1.5]{KasparovEqvar}, and $\KK(\EE)$ in particular becomes a $C_{0}(X)$-algebra when equipped with this structure.

Let us now give some standard definitions, phrased in terms of upper-semicontinuous bundles.

\begin{defn}\label{pullbax}
	Let $A = \Gamma_{0}(X;\AF)$ be a $C_{0}(X)$-algebra, and let $\EE = \Gamma_{0}(X;\EF)$ and $\EE' = \Gamma_{0}(X;\EF')$ be Hilbert $A$-modules.  The \textbf{direct sum} $\EE\oplus\EE'$ of $\EE$ and $\EE'$ is the Hilbert $A$-module $\Gamma_{0}(X;\EF\oplus\EF')$, where $\EF\oplus\EF'$ is the upper-semicontinuous Hilbert $\AF$-module bundle whose fibre over $x\in X$ is the direct sum $\EF_{x}\oplus\EF'_{x}$ in the sense of \cite[p. 518]{KasparovTech}.
	
	If $B =\Gamma_{0}(X;\BF)$ is another $C_{0}(X)$-algebra and $\FF = \Gamma_{0}(X;\mathfrak{F})$ is a Hilbert $B$-module, we say that a $C^{*}$-homomorphism $\pi:A\rightarrow\LL(\FF)$ is a $C_{0}(X)$\textbf{-representation} if it is additionally a homomorphism of $C_{0}(X)$-modules.  Given such a representation, the \textbf{balanced tensor product} $\EE\hotimes_{A}\FF$ is the Hilbert $B$-module $\Gamma_{0}(X;\EF\hotimes_{\AF}\mathfrak{F})$, where $\EF\hotimes_{\AF}\mathfrak{F}$ is the upper-semicontinuous Hilbert $\BF$-module bundle whose fibre over $x\in X$ is the balanced tensor product $\EF_{x}\hotimes_{\AF_{x}}\mathfrak{F}_{x}$ in the sense of \cite[Section 2.8]{KasparovTech}.
	
	Let $Y$ be a locally compact Hausdorff space and let $\phi:Y\rightarrow X$ be a continuous map.  The \textbf{pullback} of a $C_{0}(X)$-algebra $A = \Gamma_{0}(X;\AF)$ by $\phi$ is the $C_{0}(Y)$-algebra $\phi^{*}A:= \Gamma_{0}(Y;\phi^{*}\AF)$.  If $\EE = \Gamma_{0}(X;\EF)$ is a Hilbert $A$-module, then its pullback by $\phi$ is the Hilbert $\phi^{*}A$-module $\phi^{*}\EE:=\Gamma_{0}(Y;\phi^{*}\EF)$.	
\end{defn}

\begin{rmk}\label{rmkpullbax}\normalfont
	The balanced tensor product and direct sum of Hilbert modules given in Definition \ref{pullbax} are easily seen to coincide with the corresponding analytic notions.  For instance, any element $\xi\hotimes\eta$ of the balanced tensor product $(\EE\hotimes_{A}\FF)_{an}$ in the sense of \cite[Section 2.8]{KasparovTech} canonically determines an element $\overline{\xi\hotimes\eta}:x\mapsto\xi(x)\hotimes\eta(x)$ of $\Gamma_{0}(X;\EF\hotimes_{\AF}\mathfrak{F})$, and the resulting map $(\EE\hotimes_{A}\FF)_{an}\rightarrow\Gamma_{0}(X;\EF\hotimes_{\AF}\mathfrak{F})$ an isomorphism by \cite[Proposition C.24]{williamscrossed}.
	
	Regarding pullbacks, note that if $\phi:Y\rightarrow X$ is a continuous map of locally compact Hausdorff spaces, then $C_{0}(Y)$ carries a $C_{0}(X)$-module structure defined for $f\in C_{0}(X)$ and $g\in C_{0}(Y)$ by
	\[
	(f\cdot g)(y):=f(\phi(y))g(y),\hspace{7mm}y\in Y.
	\]
	As remarked in \cite[Section 2.7 (d)]{koshskand}, if $A = \Gamma_{0}(X;\AF)$ is a $C_{0}(X)$-algebra, then the balanced tensor product $A\hotimes_{C_{0}(X)}C_{0}(Y)$ (where $C_{0}(Y)$ is of course assumed to be trivially graded) identifies naturally with $\phi^{*}A$ via the map $a\hotimes g\mapsto g\cdot\phi^{*}a$, where the $\cdot$ denotes pointwise scalar multiplication.  Thus Definition \ref{pullbax} agrees (up to isomorphism) with the definitions used in \cite{legall,iakovos}.  In the same way, if $\EE$ is a Hilbert $A$-module, then regarding $\phi^{*}A$ as a left $A$-module via multiplication, the balanced tensor product $\EE\hotimes_{A}\phi^{*}A$ is isomorphic as a Hilbert $\phi^{*}A$-module to $\phi^{*}\EE$ via the map sending $\xi\hotimes(g\cdot\phi^{*}a)\in\EE\hotimes_{A}\phi^{*}A$ to $g\cdot \phi^{*}(\xi\cdot a)\in\phi^{*}\EE$ (where now the $\cdot$ inside the brackets denotes the right action of $A$ on $\EE$).
\end{rmk}

The final notion we will need is the pullback of an operator on a Hilbert module, which is outlined in the generality of unbounded operators in \cite[Section 2]{macr1}.  Let $A = \Gamma_{0}(X;\AF)$ be a $C_{0}(X)$-algebra, and let $\EE = \Gamma_{0}(X;\EF)$ be a Hilbert $A$-module.  Let $T:\dom(T)\rightarrow \EE$ be an $A$-linear operator on $\EE$.  For each $x\in X$, define $\mathfrak{dom}(T)_{x}:=\{\xi(x):\xi\in\dom(T)\}\subset\EF_{x}$ and let $T(x):\mathfrak{dom}(T)_{x}\rightarrow\EF_{x}$ be the $\AF_{x}$-linear operator defined by
	\[
	T(x)\xi(x):=(T\xi)(x),\hspace{7mm}\xi\in\dom(T).
	\]
Defining $\mathfrak{dom}(T):=\bigcup_{x\in X}\mathfrak{dom}(T)_{x}\subset\EF$, we see that $\dom(T)$ identifies with the subspace $\Gamma_{0}(X;\mathfrak{dom}(T))$ of $\Gamma_{0}(X;\EF)$ consisting of sections whose value at each $x\in X$ is in $\mathfrak{dom}(T)_{x}$.  The pullback of $T$ by a continuous map of locally compact Hausdorff spaces is then defined in the obvious way.

\begin{defn}\label{operatorpullback}
	Let $\phi:Y\rightarrow X$ be a continuous map of locally compact Hausdorff spaces, let $A = \Gamma_{0}(X;\AF)$ be a $C_{0}(X)$-algebra, and let $\EE = \Gamma_{0}(X;\EF)$ be a Hilbert $A$-module.  If $T:\dom(T)\rightarrow\EE$ is an $A$-linear operator, then the \textbf{pullback of $T$ by $\phi$} is the operator $\phi^{*}T:\dom(\phi^{*}T)\rightarrow\phi^{*}\EE = \Gamma_{0}(Y;\phi^{*}\EF)$ defined on the domain $\dom(\phi^{*}T):=\Gamma_{0}(Y;\phi^{*}\mathfrak{dom}(T))\subset\phi^{*}\EE$ by the formula
	\[
	\big(\phi^{*}T\,\eta\big)(y):=T(\phi(y))\eta(y)\in\EF_{\phi(y)},\hspace{7mm}\eta\in\dom(\phi^{*}T).
	\]
	In particular, if $T\in\LL(\EE)$, then $\phi^{*}T\in\LL(\phi^{*}\EE)$.
\end{defn}

Finally, let us point out if $T\in\LL(\EE)$ as in Definition \ref{operatorpullback}, then $\phi^{*}T\in\LL(\phi^{*}\EE)$ identifies via the map considered in Remark \ref{rmkpullbax} with the operator $T\hotimes 1$ on $\EE\hotimes_{A}\phi^{*}A$ that is used in \cite{legall}.

\section{Groupoid actions on algebras and modules}

We will assume for the rest of the paper that $\GG$ is a locally compact, second countable, locally Hausdorff groupoid with locally compact, Hausdorff unit space $\GG^{(0)}$.  The theory of $\GG$-algebras was originally developed by Le Gall in \cite{legall} for Hausdorff $\GG$.  In the locally Hausdorff setting it has been expounded upon in \cite{koshskand, tu, muhwil, iakovos}.  We primarily follow the bundle-theoretic picture adopted in \cite{muhwil}.

\begin{defn}\label{Galg}  Let $A = \Gamma_{0}(\GG^{(0)};\AF)$ be a $C_{0}(\GG^{(0)})$-algebra.  An action $\alpha$ of $\GG$ on $A$ consists of a family $\{\alpha_{u}\}_{u\in \GG}$ such that
	\begin{enumerate}
		\item for each $u\in \GG$, $\alpha_{u}:\AF_{s(u)}\rightarrow\AF_{r(u)}$ is a degree 0 isomorphism of $C^{*}$-algebras,
		\item the map $\GG\times_{s,p_{\AF}}\AF\rightarrow \AF$ defined by $(u,a)\mapsto\alpha_{u}(a)$ defines a continuous action of $\GG$ on $\AF$.
	\end{enumerate}
	The triple $(\AF,\GG,\alpha)$ is called a \textbf{groupoid dynamical system}, and we say that $A$ is a \textbf{$\GG$-algebra} and that it admits a \textbf{$\GG$-structure}.
\end{defn}

\begin{rmk}\normalfont
By \cite[Lemma 4.3]{muhwil}, Definition \ref{Galg} agrees with the definition given in \cite[Section 3.2]{koshskand} using pullbacks over Hausdorff open subsets of $\GG$, and agrees with \cite[Definition 3.1.1]{legall} when $\GG$ is Hausdorff.
\end{rmk}

The action of $\GG$ on a Hilbert module over a $\GG$-algebra is defined in a similar way.

\begin{defn}\label{Ghilb}
	Let $(\AF,\GG,\alpha)$ be a groupoid dynamical system, and let $\EE = \Gamma_{0}(\GG^{(0)};\EF)$ be a Hilbert $A$-module.  An action $W$ of $\GG$ on $E$ consists of a family $\{W_{u}\}_{u\in \GG}$ such that
	\begin{enumerate}
		\item for each $u\in \GG$, $W_{u}:\EF_{s(u)}\rightarrow \EF_{r(u)}$ is a degree 0, isometric isomorphism of Banach spaces such that
		\[
		\langle W_{u}e, W_{u}f\rangle_{r(u)} = \alpha_{u}\big(\langle e,f\rangle_{s(u)}\big)
		\]
		for all $e,f\in\EF_{s(u)}$, and
		\item the map $\GG\times_{s,p_{\EF}}\EF\rightarrow\EF$ defined by $(u,e)\mapsto W_{u}e$ defines a continuous action of $\GG$ on $\EF$.
	\end{enumerate}
	The tuple $(\EF,{\AF},\GG,W,\alpha)$ is called a \textbf{Hilbert module representation}.  We then say that $\EE$ is a \textbf{$\GG$-Hilbert module} and that $\EE$ admits a \textbf{$\GG$-structure}.  Conjugation by $W$ gives rise to a continuous action $\varepsilon:\GG\times_{s,p_{\LL(\EF)}}\LL(\EF)\rightarrow\LL(\EF)$ of $\GG$ on the upper semicontinuous bundle $\LL(\EF)$, which in particular makes $(\KK(\EF),\GG,\varepsilon)$ a $\GG$-algebra.
\end{defn}

\begin{rmk}\normalfont
	The arguments of \cite[Lemma 4.3]{muhwil} also show that Definition \ref{Ghilb} agrees with the notion given in \cite[Section 4.2.4]{iakovos}.
\end{rmk}

\begin{defn}
	Let $A$ and $B$ be $\GG$-algebras, with $\GG$-actions $\alpha$ and $\beta$ respectively.  We say that a homomorphism $\phi:A\rightarrow B$ is a \textbf{$\GG$-homomorphism} if
	\[
	\phi_{r(u)}\circ\alpha_{u} = \beta_{u}\circ\phi_{s(u)}
	\]
	for all $u\in\GG$.  Let $\EE$ be a $\GG$-Hilbert $B$-module, with action $\varepsilon$ of $\GG$ on $\LL(\EF)$.  We say that a $C_{0}(\GG^{(0)})$-representation is a \textbf{$\GG$-representation} if for all $u\in\GG$ we have
	\[
	\varepsilon_{u}\circ\pi_{s(u)} = \pi_{r(u)}\circ \alpha_{u}.
	\]
	Such a representation makes $\EE$ into a \textbf{$\GG$-equivariant $A$-$B$-correspondence}.
\end{defn}

Let us end this section by noting that if $A$ is a $\GG$-algebra, and $\EE$ and $\EE'$ are two $\GG$-Hilbert $A$-modules with $\GG$-structures $W$ and $W'$ respectively, then the formula
\[
(W\oplus W')_{u}:=W_{u}\oplus W'_{u}:\EF_{s(u)}\oplus\EF'_{s(u)}\rightarrow \EF_{r(u)}\oplus\EF'_{r(u)},\hspace{7mm}u\in\GG
\]
defines a $\GG$-structure $W\oplus W'$ on their direct sum $\EE\oplus\EE'$.  Similarly, if $B$ is another $\GG$-algebra and $\FF$ is an equivariant $A$-$B$-correspondence with $\GG$-structure $V$, then the formula
\[
(W\hotimes V)_{u}:=W_{u}\hotimes V_{u}:\EF_{s(u)}\hotimes_{\AF_{s(u)}}\mathfrak{F}_{s(u)}\rightarrow\EF_{r(u)}\hotimes_{\AF_{r(u)}}\mathfrak{F}_{r(u)},\hspace{7mm}u\in\GG
\]
defines a $\GG$-structure $W\hotimes V$ on the balanced tensor product $\EE\hotimes_{A}\FF$.  When considering direct sums and balanced tensor products of Hilbert modules in what follows, we will always consider them to be equipped with these $\GG$-structures without further comment. 

\section{$KK^{\GG}$-theory}

We assume from here on that all $\GG$-algebras and $\GG$-Hilbert modules are $\ZB_{2}$-graded, and carry $\GG$-actions that are of degree 0 with respect to the $\ZB_{2}$-grading.  We will moreover assume all Hilbert modules to be countably generated.  Given $\GG$-algebras $A$, $B$, the corresponding $\GG$-actions will be denoted by the corresponding lower-case Greek letters $\alpha$, $\beta$.  Given any $\GG$-Hilbert module $\EE$, the corresponding action on $\LL(\EE)$ given by conjugation will be denoted by $\varepsilon$.

Let us for the rest of this paper fix a countable base $\{U_{i}\}$ for the topology of $\GG$ consisting of locally compact, Hausdorff open subsets, and define
\[
\GG_{\sqcup}:=\bigsqcup_{i}U_{i}.
\]
For each $i$ let $\iota_{i}:U_{i}\hookrightarrow\GG_{\sqcup}$ denote the canonical (continuous) inclusion.  Then $\GG_{\sqcup}$ is a locally compact Hausdorff space.  We equip $\GG_{\sqcup}$ with the continuous maps $r_{\sqcup}:\GG_{\sqcup}\rightarrow\GG^{(0)}$ and $s_{\sqcup}:\GG_{\sqcup}\rightarrow\GG^{(0)}$ defined by
\[
r_{\sqcup}(u,i):=r(u),\hspace{7mm} s_{\sqcup}(u,i):=s(u),\hspace{7mm}(u,i)\in\GG_{\sqcup}.
\]
We will use pullbacks over $\GG_{\sqcup}$ to treat pullbacks over all Hausdorff open subsets of $\GG$ simultaneously.  Note that the action $\alpha$ of $\GG$ on any $\GG$-algebra $A$ induces an isomorphism $\alpha_{\sqcup}:s_{\sqcup}^{*}A\rightarrow r_{\sqcup}^{*}A$ of $C^{*}$-algebras defined by the formula
\[
\alpha_{\sqcup}(a'(u,i)):=\alpha_{u}\big((\iota_{i}^{*}a')(u)\big)\in\AF_{r(u)},\hspace{7mm}a'\in s_{\sqcup}^{*}A.
\]
Similarly, if $\EE$ is a $\GG$-Hilbert module, then the action $\varepsilon$ of $\GG$ by conjugation on $\LL(\EE)$ induces an isomorphism $\varepsilon_{\sqcup}:\LL(s_{\sqcup}^{*}\EE)\rightarrow\LL(r_{\sqcup}^{*}\EE)$ of $C^{*}$-algebras defined by a similar formula.

\begin{defn}\label{GEKKdefn}Let $A$ and $B$ be $\GG$-algebras.  A \textbf{$\GG$-equivariant Kasparov $A$-$B$-module} is a tuple $(\EE,F)$, where $\EE$ is a $\GG$-equivariant $A$-$B$ correspondence, and where $F\in\LL(\EE)$ is homogeneous of degree 1 such that for all $a\in A$ one has
	\begin{enumerate}
		\item $a(F-F^{*})\in\KK(\EE)$,
		\item $a(F^{2}-1)\in\KK(\EE)$,
		\item $[F,a]\in\KK(\EE)$, and
		\item $a'\big(\varepsilon_{\sqcup}(s_{\sqcup}^{*}F)-r_{\sqcup}^{*}F\big)\in r_{\sqcup}^{*}\KK(\EE)$ for all $a'\in r_{\sqcup}^{*}A$.
	\end{enumerate}
	We will sometimes refer to item 4. by saying that $F$ is \textbf{almost-equivariant}.  We say that two $\GG$-equivariant Kasparov $A$-$B$-modules $(\EE,F)$ and $(\EE',F')$ are \textbf{unitarily equivalent} if there exists a degree 0 unitary isomorphism $V:\EE\rightarrow \EE'$ of $\GG$-equivaraint $A$-$B$-correspondences for which $VF = F'V$.  We denote by $\EB^{\GG}(A,B)$ the set of all unitary equivalence classes of $\GG$-equivariant Kasparov $A$-$B$-modules.
\end{defn}

\begin{rmk}\normalfont\label{nonequi}
	Notice of course that in the case where $\GG$ is simply a point, $\EB^{\GG}(A,B)$ coincides with the set $\EB(A,B)$ of non-equivariant $KK$-theory \cite[Definition 1]{KasparovTech}.
\end{rmk}

\begin{rmk}\normalfont
	Definition \ref{GEKKdefn} finds its origin in \cite[Section 4.3.4]{iakovos}, although the idea of using a disjoint union to treat locally Hausdorff groupoids dates back at least as early as \cite{koshskand2}.  It differs from that usually seen in the literature (for instance, \cite[Definition 10.10]{tu}, \cite[Definition 4.6]{iakovos}, or \cite[Definition 2.7]{macr1}) in which one replaces item 4. with the requirement that for each Hausdorff open subset $U$ of $\GG$, one has
	\begin{equation}\label{usualrequirement}
	a'\big(\varepsilon_{U}(s|_{U}^{*}F)-r|_{U}^{*}F\big)\in r|_{U}^{*}\KK(\EE)
	\end{equation}
	for all $a'\in r|_{U}^{*}A$.  It is clear that the usual requirement \eqref{usualrequirement} implies item 4. in Definition \ref{GEKKdefn}.  On the other hand, by pulling back via the inclusions $\iota_{i}:U_{i}\hookrightarrow\GG_{\sqcup}$, it is easily checked that any $\GG$-equivariant Kasparov module in the sense of Definition \ref{GEKKdefn} is an equivariant Kasparov module in the sense of \eqref{usualrequirement}.  The advantage of Definition \ref{GEKKdefn} is that it allows us to treat all Hausdorff open subsets of $\GG$ at once, resulting in much cleaner notation and easy extension of the ``classical" proofs to the equivariant context.
\end{rmk}

Notice that if $B$ is a $\GG$-algebra, then the tensor product $B\hotimes C([0,1])$ (with $C([0,1])$ trivially graded) is canonically a $\GG$-algebra - as a $C_{0}(\GG^{(0)})$ algebra, it has fibre $\BF_{x}\hotimes C([0,1])$ over each $x\in\GG^{(0)}$, and the action $\beta$ of $\GG$ on $B$ extends trivially to the $C([0,1])$-factor to give an action on $B\hotimes C([0,1])$.  For $t\in[0,1]$, let $\ev_{t}:C([0,1])\rightarrow\CB$ denote the evaluation functional $\ev_{t}(f):=f(t)$.  Then to each $t\in[0,1]$ and each $\GG$-equivariant Kasparov $A$-$B\hotimes C([0,1])$-module $(\EE,F)$ is associated the $\GG$-equivariant Kasparov $A$-$B$ module $(\ev_{t})_{*}(\EE,F):=(\EE\hotimes_{B\hotimes C([0,1])}B,F\hotimes 1)$, where $B$ is regarded as a left $B\hotimes C([0,1])$-module via the left action $(b\hotimes f)\cdot b':=\ev_{t}(f)bb'$.

\begin{defn}\label{kk2}
	Two $\GG$-equivariant Kasparov $A$-$B$-modules $(\EE_{0},F_{0})$ and $(\EE_{1},F_{1})$ are said to be $\GG$-\textbf{homotopic}, written $(\EE_{0},F_{0})\sim_{h}(\EE_{1},F_{1})$, if there exists a $\GG$-equivariant Kasparov $A$-$B\hotimes C([0,1])$ module $(\EE,F)$ such that $(\ev_{i})_{*}(\EE,F) = (\EE_{i},F_{i})$ for $i =0, 1$.  The relation $\sim_{h}$ is an equivalence relation, and we denote by $KK^{\GG}(A,B)$ the set of $\sim_{h}$-equivalence classes of elements in $\EB^{\GG}(A,B)$.
\end{defn}

The following analogue of \cite[Lemma 11]{sk1} is essential for the uniqueness and associativity of the Kasparov product at the level of the $KK^{\GG}$ groups.  Its proof in the equivariant setting does not currently appear in the literature.

\begin{lemma}\label{homo}
	Let $A$ and $B$ be $\GG$-algebras, and let $(\EE,F)$, $(\EE,F')\in \EB^{\GG}(A,B)$.  If for all $a\in A$ one has $a[F,F']a^{*}\geq0$ modulo $\KK(\EE)$, then $(\EE,F)$ and $(\EE,F')$ are $\GG$-homotopic.
\end{lemma}

\begin{proof}
	By the arguments of \cite[Lemma 11]{sk1}, we obtain a positive operator $P\geq0$ for which $[P,a]\in\KK(\EE)$ for all $a\in A$, and an operator $K$ for which $Ka\in\KK(\EE)$ for all $a\in A$, such that
	\[
	[F,F'] = P+K.
	\]
	As in \cite[Lemma 11]{sk1}, for each $t\in[0,\pi/2]$ we define the operator
	\[
	F_{t}:=(1+\cos(t)\sin(t)P)^{-\frac{1}{2}}(\cos(t)F+\sin(t)F'),
	\]
	and the pair $(\EE,F_{t})$ satisfies items 1., 2., and 3. of Definition \ref{GEKKdefn}.  All that remains to check to complete the proof of the lemma is that each $F_{t}$ is almost-equivariant.
	
	For $t\in[0,\pi/2]$, write $F_{t1}:=(1+\cos(t)\sin(t)P)^{-\frac{1}{2}}$, and $F_{t2} = (\cos(t)F+\sin(t)F')$.  Note then that for $a\in r_{\sqcup}^{*}A$, we can write
	\begin{equation}\label{eqnsk}
	a\varepsilon_{\sqcup}(s_{\sqcup}^{*}F_{t})-r_{\sqcup}^{*}F_{t} = a\big(\varepsilon_{\sqcup}(s_{\sqcup}^{*}F_{t1})-r_{\sqcup}^{*}F_{t1}\big)\varepsilon_{\sqcup}(s_{\sqcup}^{*}F_{t2})+ar_{\sqcup}^{*}(F_{t1})\big(\varepsilon_{\sqcup}(s_{\sqcup}^{*}F_{t2})-r_{\sqcup}^{*}F_{t2}\big).
	\end{equation}
	The second of these terms is contained in $\KK(\EE)$ by the almost-equivariance of $F_{1}$ and $F_{2}$, together with the fact that $P$ commutes with $A$ up to $\KK(\EE)$.
	
	To show that the first term is contained in $\KK(\EE)$, observe first that, modulo $\KK(\EE)$, we have
	\begin{align*}
	a(\varepsilon_{\sqcup}(s_{\sqcup}^{*}P)-r^{*}_{\sqcup}P) =& a(\varepsilon_{\sqcup}(s_{\sqcup}^{*}[F,F'])-r_{\sqcup}^{*}[F,F'])\\ =& a\big((\varepsilon_{\sqcup}(s_{\sqcup}^{*}F)-r_{\sqcup}^{*}F)\varepsilon_{\sqcup}(s_{\sqcup}^{*}F')+r_{\sqcup}^{*}F\,(\varepsilon_{\sqcup}(s_{\sqcup}^{*}F')-r_{\sqcup}^{*}F')\\ &+ (\varepsilon_{\sqcup}(s_{\sqcup}^{*}F')-r_{\sqcup}^{*}F')\varepsilon_{\sqcup}(s_{\sqcup}^{*}F)+r_{\sqcup}^{*}F'\,(\varepsilon_{\sqcup}(s_{\sqcup}^{*}F)-r_{\sqcup}^{*}F)\big)
	\end{align*}
	which is contained in $\KK(\EE)$ by the equivariance properties of $F$ and $F'$ together with the fact that both $F$ and $F'$ commute up to $\KK(\EE)$ with $A$.  That this implies that the first term in Equation \eqref{eqnsk} is contained in $\KK(\EE)$ now follows from an elegant integral argument, which was brought to our attention by A. Rennie.
	
	Denote $\cos(t)\sin(t)$ by $c(t)$.  Using the Laplace transform, we have the norm-convergent integral formula
	\[
	\varepsilon_{\sqcup}(s_{\sqcup}^{*}F_{t1})-r_{\sqcup}^{*}F_{t1} = \frac{1}{\sqrt{\pi}}\int_{0}^{\infty}\lambda^{-\frac{1}{2}}\big(\exp\big(-\lambda(1+c(t)\varepsilon_{\sqcup}(s_{\sqcup}^{*}P))\big)-\exp\big(-\lambda\big(1+c(t)r_{\sqcup}^{*}P)\big)d\lambda.
	\]
	By the fundamental theorem of calculus, this becomes
	\[
	\frac{1}{\sqrt{\pi}}\int_{0}^{\infty}\lambda^{-\frac{1}{2}}\int_{0}^{1}\frac{d}{ds}\exp\big(-\lambda\big(1+c(t)(s\varepsilon_{\sqcup}(s_{\sqcup}^{*}P)+(1-s)r_{\sqcup}^{*}P))\big)ds\,d\lambda,
	\]
	which we easily compute to be
	\[
	-\frac{c(t)}{\sqrt{\pi}}(\varepsilon_{\sqcup}(s_{\sqcup}^{*}P)-r_{\sqcup}^{*}P)\int_{0}^{\infty}\lambda^{\frac{1}{2}}\int_{0}^{1}\exp\big(-\lambda(1+c(t)(s\varepsilon_{\sqcup}(s_{\sqcup}^{*}P)+(1-s)r_{\sqcup}^{*}P)\big)\,ds\,d\lambda.
	\]
	Since $a(\varepsilon_{\sqcup}(s_{\sqcup}^{*}P)-r^{*}_{\sqcup}P)\in\KK(\EE)$, it follows that the first term in Equation \eqref{eqnsk} is contained in $\KK(\EE)$, hence that $(\EE,F_{t})\in\EB^{\GG}(A,B)$ as required.
\end{proof}

As we would expect, $KK^{\GG}(A,B)$ is indeed an abelian group for all $\GG$-algebras $A$ and $B$.  In order to prove this, we must define the appropriate notion of degeneracy for our setting.

\begin{defn}\label{degen}
	An element $(\EE,F)$ of $\EB^{\GG}(A,B)$ is said to be \textbf{degenerate} if for all $a\in A$ one has
	\begin{enumerate}
		\item $a(F^{2}-1) = 0$,
		\item $[F,a] = 0$, and
		\item $\varepsilon_{\sqcup}(s_{\sqcup}^{*}F)-r_{\sqcup}^{*}F = 0$.
	\end{enumerate}
\end{defn}

For $\GG$-algebras $A$ and $B$, we define the sum $(\EE\oplus\EE',F\oplus F')$ of two elements $(\EE,F)$ and $(\EE',F')$ in $\EB^{\GG}(A,B)$ in the same way as in \cite[Definition 3]{KasparovTech}.  The natural operator
\[
F\oplus F':=\begin{pmatrix} F & 0 \\ 0 & F' \end{pmatrix}\in\LL(\EE\oplus\EE')
\]
satisfies item 4. of Definition \ref{GEKKdefn} by the corresponding properties of $F$ and $F'$, so that $(\EE\oplus\EE',F\oplus F')\in\EB^{\GG}(A,B)$.

\begin{prop}
	Under the sum of equivariant Kasparov modules, $KK^{\GG}(A,B)$ is an abelian group.
\end{prop}

\begin{proof}
	For $(\EE,F)\in\EB^{\GG}(A,B)$, the explicit homotopy given in \cite[Theorem 1]{KasparovTech} from $(\EE\oplus(-\EE),F\oplus(-F))$ to a degenerate module is (up to a reparameterisation of $[0,\pi/2]$ onto $[0,1]$) easily seen to be a $\GG$-homotopy by the almost-equivariance of $F$.
\end{proof}

Functoriality of the equivariant Kasparov groups goes through just as in the classical case.

\begin{prop}
	Let $A$, $B$, and $C$ be $\GG$-algebras.
	\begin{enumerate}
		\item any $\GG$-homomorphism $\phi:A\rightarrow C$ determines a homomorphism
		\[
		\phi^{*}:KK^{\GG}(C,B)\rightarrow KK^{\GG}(A,B),
		\]
		of abelian groups,
		\item any $\GG$-homomorphism $\psi:B\rightarrow C$ determines a homomorphism
		\[
		\psi_{*}:KK^{\GG}(A,B)\rightarrow KK^{\GG}(A,C)
		\]
		of abelian groups, and
		\item any homomorphism $\varphi:\HH\rightarrow\GG$ of locally compact, second-countable, locally Hausdorff groupoids determines a homomorphism
		\[
		\varphi^{*}:KK^{\GG}(A,B)\rightarrow KK^{\HH}(\varphi^{*}A,\varphi^{*}B)
		\]
		of abelian groups, where $\varphi^{*}A = \Gamma_{0}(\HH^{(0)};\varphi^{*}\AF)$ and $\varphi^{*}B = \Gamma_{0}(\HH^{(0)};\varphi^{*}\BF)$ are regarded as $\HH$-algebras using $\varphi$ in the obvious way. 
	\end{enumerate}
\end{prop}

\begin{proof}
	The maps $\phi^{*}$ and $\psi_{*}$ are defined at the level of Kasparov modules just as in \cite[Definition 4]{KasparovTech}.  That they send equivariant Kasparov modules to equivariant Kasparov modules follows in a routine manner from the equivariance of the maps $\phi$ and $\psi$.  For the third item, given $(\EE,F)\in\EB^{\GG}(A,B)$ we note that $\varphi^{*}(\EE,F):=(\Gamma_{0}(\HH^{(0)};\varphi^{*}\EF),\varphi^{*}F)$ defines an element of $\EB^{\HH}(\varphi^{*}A,\varphi^{*}B)$ because $\varphi$ is a homomorphism, and $(\EE,F)\mapsto\varphi^{*}(\EE,F)$ determines the stated map $\varphi^{*}$ on $KK$-groups.
\end{proof}

We end the section with a final definition and result on unbounded representatives, which are slight modifications of those found in \cite{pierrot}.  For a $\GG$-algebra $A$, we say that a (not necessarily norm-closed) subalgebra $\AA\subset A$ is a \emph{$*$-$\GG$-subalgebra} if there exists some subspace $\AS\subset \AF$ which is preserved by the action of $\GG$, for which $\AS_{x}$ is a subalgebra of $\AF_{x}$, and for which $\AA$ can be written as a subalgebra of $\Gamma_{0}(\GG^{(0)};\AF)$ whose elements take values in $\AS$.  For instance, when $\GG$ is a Lie groupoid and $P\rightarrow\GG^{(0)}$ is a smooth submersion carrying a smooth $\GG$-action, one sees that $A = C_{0}(P)$ is a $\GG$-algebra, whose fibre over $x\in\GG^{(0)}$ is $C_{0}(P_{x})$.  In this case, $\AA:=C_{0}^{\infty}(P)$ is a $*$-$\GG$-subalgebra of $A$, with corresponding fibre $\AS_{x}:=C_{0}^{\infty}(P_{x})\subset C_{0}(P_{x})$ for each $x\in X$.

\begin{defn}\label{unbdd}
	Let $A$ and $B$ be $\GG$-algebras.  An \textbf{unbounded $\GG$-equivariant Kasparov $A$-$B$-module} is a triple $(\AA,\EE,D)$, where $\AA$ is a dense $*$-$\GG$-subalgebra of $A$, $\EE$ is a $\GG$-equivariant $A$-$B$-correspondence, and where $D$ is a densely defined, $B$-linear, self-adjoint and regular operator on $\EE$ of degree 1 such that
	\begin{enumerate}
		\item $a(1+D^{2})^{-\frac{1}{2}}\in\KK(E)$ for all $a\in A$,
		\item for all $a\in\AA$, the operator $[D,a]$ extends to an element of $\LL(E)$, and for $f\in C_{c}(\GG_{\sqcup})$ one has
		\[
		f\,r_{\sqcup}^{*}a\,\big(\varepsilon_{\sqcup}(s_{\sqcup}^{*}D)-r_{\sqcup}^{*}D\big)\in\LL(r_{\sqcup}^{*}\EE),
		\]
		and
		\item for all $f\in C_{c}(\GG_{\sqcup})$, one has $\dom(r_{\sqcup}^{*}D\,f) = W_{\sqcup}\dom(s_{\sqcup}^{*}D \,f)$, where $W_{\sqcup}:s_{\sqcup}^{*}\EE\rightarrow r_{\sqcup}^{*}\EE$ is the isometric isomorphism of Banach spaces induced by the action of $\GG$ on $\EE$.
	\end{enumerate}
\end{defn}

\begin{prop}
	Let $A$ and $B$ be $\GG$-algebras.  Every unbounded equivariant Kasparov $A$-$B$-module $(\AA,\EE,D)$ determines a $\GG$-equivariant Kasparov $A$-$B$-module $(\EE,D(1+D^{2})^{-\frac{1}{2}})$, defining a class $[D]\in KK^{\GG}(A,B)$.
\end{prop}

\begin{proof}
	By the same arguments as in the non-equivariant case \cite{baajjulg1}, the pair $(\EE,D(1+D^{2})^{-\frac{1}{2}})$ satisfies items 1., 2. and 3. of Definition \ref{GEKKdefn}.  The final item 4. in Definition \ref{GEKKdefn} follows from \cite[Th\'{e}or\`{e}me 6]{pierrot}, with Pierrot's $G$ replaced with our $\GG_{\sqcup}$.
\end{proof}

\section{The Kasparov product}

The characteristic feature of $KK$-theory is the Kasparov product.  In \cite{legall}, it is proved that when $\GG$ is Hausdorff, for all $\GG$-algebras $A$, $D$, $B$ there is an associative and non-trivial product
\[
KK^{\GG}(A,D)\otimes KK^{\GG}(D,B)\rightarrow KK^{\GG}(A,B).
\]
The same is true when $\GG$ is locally Hausdorff.  In fact, using the ideas of \cite[Section 4.3.4]{iakovos}, this can be seen by direct substitution of concepts into the proofs used in \cite{KasparovEqvar}.  We first give the groupoid-equivariant version of the the lemma \cite[Lemma 1.4]{KasparovEqvar} on the existence of quasi-invariant, quasi-central approximate units.

\begin{lemma}[Existence of quasi-central approximate units]\label{qau}
Let $B$ be a $\GG$-algebra, $A$ a $\sigma$-unital $\GG$-subalgebra of $B$, $Y$ a $\sigma$-compact, locally compact Hausdorff space, and $\varphi:Y\rightarrow B$ a function satisfying:
	\begin{enumerate}
		\item $[\varphi(y),a]\in A$ for all $a\in A$ and $y\in Y$, and
		\item for all $a\in A$, the functions $y\mapsto[\varphi(y),a]$ are norm-continuous on $Y$.
	\end{enumerate}
Then there is a countable approximate unit $\{e_{i}\}$ for $A$ with the following properties:
	\begin{enumerate}
		\item $\lim_{i\rightarrow\infty}\|[e_{i},\varphi(y)]\| = 0$ for all $y\in Y$, and
		\item $\lim_{i\rightarrow\infty}\|\alpha_{\sqcup}(s_{\sqcup}^{*}e_{i})-r_{\sqcup}^{*}e_{i}\| = 0$ in $r_{\sqcup}^{*}A$.
	\end{enumerate}
\end{lemma}

\begin{proof}
	The proof given in \cite[Lemma 1.4]{KasparovEqvar} applies immediately.  One need only replace Kasparov's $G$ with our $\GG_{\sqcup}$, and Kasparov's $g(u_{i})-u_{i}$, defined for $g\in G$, with our $\alpha_{u}(e_{i}(s(u)))-e_{i}(r(u))$, defined for $u\in \GG$.
\end{proof}

Kasparov's Technical Theorem below now follows essentially ``classically".

\begin{thm}[Kasparov's Technical Theorem]\label{kastechlem}
Let $J$ be a $\sigma$-unital $\GG$-algebra, $A_{1}$ and $A_{2}$ $\sigma$-unital subalgebras of $\MM(J)$, with $A_{1}$ a $\GG$-algebra.  Let $\Delta$ be a subset of $\MM(J)$ which is separable in the norm topology and whose commutators with $A_{1}$ act as derivations of $A_{1}$, and let $\varphi\in\MM(r_{\sqcup}^{*}J)$.  Assume that $A_{1}\cdot A_{2}\subset J$, $r_{\sqcup}^{*}A_{1}\cdot\varphi\subset r_{\sqcup}^{*}J$, and that $\varphi\cdot r_{\sqcup}^{*}A_{1}\subset r_{\sqcup}^{*}J$.  Then there are $M_{1},\, M_{2}\in \MM(J)$ of degree 0 such that $M_{1}+M_{2} = 1$, and for which
\begin{enumerate}
	\item $M_{i}a_{i}\in J$ for all $a_{i}\in A_{i}$ and $i = 1,2$,
	\item $[M_{i},d]\in J$ for all $d\in\Delta$, and
	\item $r_{\sqcup}^{*}(M_{2})\cdot\varphi$, $\varphi\cdot r_{\sqcup}^{*}(M_{2})$ and $\alpha_{\sqcup}(s_{\sqcup}^{*}(M_{i})) - r_{\sqcup}^{*}(M_{i})$ are all contained in $r_{\sqcup}^{*}(J)$, for $i = 1,2$.
\end{enumerate}
\end{thm}

\begin{proof}
	After making the same replacements as in the proof of Lemma \ref{qau}, the proof of \cite[Theorem 1.4]{KasparovEqvar} (which in this form is due to Higson \cite{higsonkastech}) applies without change.
\end{proof}

\begin{rmk}\normalfont
	Note that by hypothesis, our $\varphi\in\MM(r_{\sqcup}^{*}(J))$ in Theorem \ref{kastechlem} already gives, for each Hausdorff open subset $U$ of $\GG$, a section $\varphi_{U}\in\MM(r|_{U}^{*}J)$ which is norm-continuous over $U$.  This is to be contrasted with Kasparov's weaker hypotheses on his $\varphi:G\rightarrow \MM(J)$ in \cite[Theorem 1.4]{KasparovEqvar}, where he requires only that $\varphi$ be bounded and induce norm-continuous functions $g\mapsto a\varphi(g)$ and $g\mapsto\varphi(g)a$ for all $a\in A_{1}+J$.  As we will see, our stronger hypothesis does not impact the proof of the existence of the Kasparov product in our setting.
\end{rmk}

We recall here the notion of a connection given by Connes and Skandalis \cite{cs}.  Let $A$ and $B$ be $C^{*}$-algebras, and suppose that $\EE_{1}$ is a Hilbert $A$-module, $\EE_{2}$ is a Hilbert $B$-module, and that $\pi:A\rightarrow\LL(\EE_{2})$ is a representation.  Let $\EE_{12} = \EE_{1}\hotimes_{A}\EE_{2}$, and for each $\xi\in \EE_{1}$ we denote by $T_{\xi}\in\LL(\EE_{2},\EE_{12})$ defined by
\[
T_{\xi}\eta:=\xi\hotimes\eta,\hspace{7mm}\eta\in \EE_{2}.
\]
The adjoint of $T_{\xi}$ is given on $\eta\hotimes\zeta\in \EE_{12}$ by
\[
T_{\xi}^{*}(\eta\hotimes\zeta) = \pi(\langle\xi,\eta\rangle)\zeta.
\]
If $F_{2}\in\LL(\EE_{2})$, we say that an operator $F\in\LL(\EE_{12})$ is an \emph{$F_{2}$-connection for $\EE_{1}$} if for all $\xi\in \EE_{1}$, one has
\[ 
T_{\xi}F_{2}-(-1)^{\deg(\xi)\deg(F_{2})}FT_{\xi}\in\KK(\EE_{2},\EE_{12}),
\]
and
\[
F_{2}T_{\xi}^{*}- (-1)^{\deg(\xi)\deg(F_{2})}T_{\xi}^{*}F\in\KK(\EE_{12},\EE_{2}).
\]
If $\EE_{1}$ is countably generated and $[F_{2},\pi(A)]\subset\KK(\EE_{2})$, then the algebra $\LL(\EE_{12})$ contains an $F_{2}$-connection for $\EE_{1}$ \cite[Proposition A.2]{cs}.

We can now define the Kasparov product of two equivariant Kasparov modules in an analogous fashion to the non-equivariant case.

\begin{defn}\label{kasproddefn}
 Let $A$, $B$ and $D$ be $\GG$-algebras, with $A$ separable, $(\EE_{1},F_{1})\in\EB^{\GG}(A,D)$, and $(\EE_{2},F_{2})\in\EB^{\GG}(D,B)$.  Denote by $\EE_{12}$ the bimodule $\EE_{1}\hotimes_{D}\EE_{2}$.  A pair $(\EE_{12},F)$, where $F\in\LL(\EE_{12})$, is called a \textbf{Kasparov product} of $(\EE_{1},F_{1})$ and $(\EE_{2},F_{2})$, written $F\in F_{1}\sharp_{D} F_{2}$, if and only if
\begin{enumerate}
	\item $(\EE_{12},F)\in\EB^{\GG}(A,B)$,
	\item $F$ is an $F_{2}$-connection, and
	\item for all $a\in A$, $a[F_{1}\hotimes 1,F] a^{*}\geq0$ modulo $\KK(\EE_{12})$.
\end{enumerate}
\end{defn}

\begin{thm}\label{kasprod}
	Let $A$, $B$ and $D$ be $\GG$-algebras, with $A$ separable, and suppose $(\EE_{1},F_{1})\in\EB^{\GG}(A,D)$ and $(\EE_{2},F_{2})\in\EB^{\GG}(D,B)$.  Let $\EE_{12}:=\EE_{1}\hotimes_{D}\EE_{2}$.  Then $F_{1}\sharp_{D}F_{2}$ is nonempty, and path connected.  Moreover the class in $KK^{\GG}(A,B)$ determined by any $F\in F_{1}\sharp_{D} F_{2}$ depends only on the $\GG$-homotopy classes of $(\EE_{1},F_{1})$ and $(\EE_{2},F_{2})$ in $KK^{\GG}(A,D)$ and $KK^{\GG}(D,B)$ respectively, and the Kasparov product descends to an associative, bilinear product $KK^{\GG}(A,D)\otimes KK^{\GG}(D,B)\rightarrow KK^{\GG}(A,B)$.
\end{thm}

\begin{proof}
Choose an $F_{2}$-connection $\tilde{F}_{2}$ for $\EE_{1}$.  Let $J$ denote the algebra $\KK(\EE_{12})$, $A_{1}$ the $\GG$-algebra $J+\KK(\EE_{1})\hotimes1$, $A_{2}$ the algebra generated by $\tilde{F}_{2}-\tilde{F}_{2}^{*}$, $\tilde{F}_{2}^{2}-1$, $[\tilde{F}_{2},F_{1}\hotimes 1]$ and $[\tilde{F}_{2},A]$, and let $\Delta$ denote the subspace of $\LL(\EE_{12})$ generated by $F_{1}\hotimes 1$, $\tilde{F}_{2}$ and $A$.  Finally denote by $\varphi$ the element $\alpha_{\sqcup}(s_{\sqcup}^{*}\tilde{F}_{2})-r_{\sqcup}^{*}\tilde{F}_{2}$ of $\LL(r_{\sqcup}^{*}\EE_{12})$.  Apply Theorem \ref{kastechlem} to obtain $M_{1}$ and $M_{2}$ with the stated properties, and define
\[
F:=M_{1}^{\frac{1}{2}}(F_{1}\hotimes 1)+M_{2}^{\frac{1}{2}}\tilde{F}_{2}.
\]
Exactly as in the non-equivariant case, the pair $(\EE_{12},F)$ satisfies properties 1., 2., and 3., in Definition \ref{GEKKdefn}, and all that remains to verify that $(\EE_{12},F)\in\EB^{\GG}(A,B)$ is to show that $F$ is almost-equivariant.  For any $a\in r_{\sqcup}^{*}A$, however, we can write $a\cdot(\alpha_{\sqcup}(s_{\sqcup}^{*}F)-r_{\sqcup}^{*}F)\in\LL(r_{\sqcup}^{*}\EE_{12})$ as the sum
\begin{align}\label{sums}
&a\cdot(\alpha_{\sqcup}(s_{\sqcup}^{*}(M_{1}^{\frac{1}{2}}))-r_{\sqcup}^{*}(M_{1}^{\frac{1}{2}}))\alpha_{\sqcup}(s_{\sqcup}^{*}(F_{1}\hotimes 1))+a\cdot r_{\sqcup}^{*}(M_{1}^{\frac{1}{2}})(\alpha_{\sqcup}(s_{\sqcup}^{*}(F_{1}\hotimes 1)-r_{\sqcup}^{*}(F_{1}\hotimes 1)) \nonumber\\ +& a\cdot(\alpha_{\sqcup}(s_{\sqcup}^{*}(M_{2}^{\frac{1}{2}}))-r_{\sqcup}^{*}(M_{2}^{\frac{1}{2}}))\alpha_{\sqcup}(s_{\sqcup}^{*}(\tilde{F}_{2}))+a\cdot r_{\sqcup}^{*}(M_{2}^{\frac{1}{2}})(\alpha_{\sqcup}(s_{\sqcup}^{*}(\tilde{F}_{2}))-r_{\sqcup}^{*}(\tilde{F}_{2}))
\end{align}
The first, third and final terms in Equation \eqref{sums} are contained in $\KK(\EE_{12})$ by item 3. of Theorem \ref{kastechlem}, while the second term in Equation \eqref{sums} is contained in $\KK(\EE_{12})$ by the equivariance of the Kasparov module $(\EE_{1},F_{1})$ together with items 1. and 2. of Theorem \ref{kastechlem}.  Thus $(\EE_{12},F)\in\EB^{\GG}(A,B)$.

Items b. and c. of Definition \ref{kasproddefn} hold by the same arguments as in the non-equivariant case \cite[Theorem 12]{sk1} (cf. \cite[Theorem A.5]{cs}).  This gives existence.  Path connectedness of $F_{1}\sharp_{D}F_{2}$ follows from the same argument given in \cite[Theorem 12]{sk1}, using Lemma \ref{homo} in the place of Skandalis' \cite[Lemma 11]{sk1}, giving uniqueness up to $\GG$-homotopy.  That the Kasparov product descends to a well-defined bilinear product on the $KK^{\GG}$ groups also follows from the arguments of \cite[Theorem 12]{sk1}.  Associativity of this product follows from \cite[Lemma 22]{sk1}, again using Lemma \ref{homo} in the place of \cite[Lemma 11]{sk1}.
\end{proof}

\begin{rmk}\normalfont
	With Theorem \ref{kasprod} in hand, one can now define the external product $KK^{\GG}(A_{1},B_{1}\hotimes D)\otimes KK^{\GG}(D\hotimes A_{2},B_{2})\rightarrow KK^{\GG}(A_{1}\hotimes A_{2},B_{1}\hotimes B_{2})$ in the same manner as \cite[Definition 2.12]{KasparovEqvar}.
\end{rmk}

\section{Crossed products}

We now discuss crossed products of $\GG$-algebras as well as their relationship to $KK^{\GG}$-theory.  We will assume for this that $\GG$ is equipped with a \emph{Haar system}, whose definition we recall for the reader's convenience.

\begin{defn}\label{haar}
	A \textbf{Haar system on $\GG$} is a family $\{\lambda^{x}\}_{x\in\GG^{(0)}}$ of measures on $\GG$ such that each $\lambda^{x}$ is supported on $\GG^{x}$, and is a regular Borel measure thereon, and such that
	\begin{enumerate}
		\item for all $v\in\GG$ and $f\in C_{c}(\GG)$ one has
		\[
		\int_{\GG}f(vu)\,d\lambda^{s(v)}(u) = \int_{\GG}f(u)\,d\lambda^{r(v)}(u)
		\]
		and,
		\item for each $f\in C_{c}(\GG)$, the map
		\[
		x\mapsto \int_{\GG}f(u)\,d\lambda^{x}(u)
		\]
		is continuous with compact support on $\GG^{(0)}$.
	\end{enumerate}
\end{defn}

That $\GG$ admits a Haar system is a nontrivial requirement in general.  Indeed, if $\GG$ admits a Haar system then its range and source maps must be open \cite[Proposition 2.2.1]{pat}, and it is not difficult to find examples of locally compact (even Hausdorff) groupoids for which this is not the case \cite[Section 3]{seda2}.  In the context of a foliated manifold, any choice of leafwise half-density (see \cite[Section 5]{folops}) determines a Haar system on the associated holonomy groupoid.

If $p_{\AF}:\AF\rightarrow\GG^{(0)}$ is an upper-semicontinuous Banach bundle, then for any Hausdorff open subset $U$ of $\GG$, we denote by $\Gamma_{c}(U;r^{*}\AF)$ the space of compactly supported continuous sections of the bundle $r^{*}\AF$ over $U$, extended by zero outside of $U$ to a (not necessarily continuous) function on $\GG$.  We define $\Gamma_{c}(\GG;r^{*}\AF)$ to be the subspace of sections of $r^{*}\AF$ over $\GG$ that are spanned by elements of $\Gamma_{c}(U;r^{*}\AF)$ as $U$ varies over all Hausdorff open subsets of $\GG$.

\begin{rmk}\normalfont
	If $\GG$ is a Lie groupoid (such as the holonomy groupoid of a foliation), then one will usually take smooth sections, denoted $\Gamma_{c}^{\infty}$, instead of continuous sections.
\end{rmk}

\begin{prop}\cite[Proposition 4.4]{muhwil}
	If $(\AF,\GG,\alpha)$ is a groupoid dynamical system, the space $\Gamma_{c}(\GG;r^{*}\AF)$ is a $*$-algebra with respect to the convolution product and adjoint given respectively by
	\[
	f*g(u):=\int_{\GG}f(v)\alpha_{v}\big(g(v^{-1}u)\big)\,d\lambda^{r(u)}(v)\hspace{7mm} f^{*}(u):=\alpha_{u}(f(u^{-1})^{*})
	\]
	for all $f,g\in\Gamma_{c}(\GG;r^{*}\AF)$ and $u\in\GG$.\qed
\end{prop}

\begin{rmk}\label{convention}\normalfont
	Note that one could also define a $*$-algebra structure on the space $\Gamma_{c}(\GG;s^{*}\AF)$ obtained by pulling back $\AF$ via the source map.  In this case, the multiplication and involution are given by
	\[
	f*g(u):=\int_{\GG}\alpha_{w}\big(f(uw)\big)g(w^{-1})\,d\lambda^{s(u)}(w),\hspace{7mm} f^{*}(u):=\alpha_{u^{-1}}\big(f(u^{-1})^{*}\big).
	\]
	Moreover, the map $f\mapsto \alpha\circ f$ defines an isomorphism $\Gamma_{c}(\GG;s^{*}\AF)\rightarrow\Gamma_{c}(\GG;r^{*}\AF)$ of $*$-algebras.  While $\Gamma_{c}(\GG;s^{*}\AF)$ is in a certain sense the more natural object to consider for \emph{left} actions of groupoids, we choose to work with pullbacks over the range in accordance with \cite{macr1}.
\end{rmk}

We now obtain the full crossed product of a $\GG$-algebra as follows (cf. \cite[Section 3.5]{koshskand}, \cite[p. 23]{muhwil}).

\begin{defn}\label{fullcrossed}
	If $(\AF,\GG,\alpha)$ is a groupoid dynamical system, one defines the \textbf{$I$-norm} of an element $f\in\Gamma_{c}(\GG;r^{*}\AF)$ by
	\[
	\|f\|_{I}:=\max\bigg\{\sup_{x\in\GG^{(0)}}\int_{\GG}\|f(u)\|_{x}d\lambda^{x}(u),\,\sup_{x\in\GG^{(0)}}\int_{\GG}\|f(u)\|_{r(u)}d\lambda_{x}(u)\bigg\}.
	\]
	The $I$-norm makes $\Gamma_{c}(\GG;r^{*}\AF)$ into a normed $*$-algebra, and the \textbf{full crossed product} $A\rtimes\GG$ is the associated enveloping $C^{*}$-algebra.  More specifically, $A\rtimes\GG$ is the completion of $\Gamma_{c}(\GG;r^{*}\AF)$ in the norm
	\[
	\|f\|_{A\rtimes\GG}:=\sup\{\|\pi(f)\|:\text{$\pi$ is a $\|\cdot\|_{I}$-decreasing $*$-representation of $\Gamma_{c}(\GG;r^{*}\AF)$}\}.
	\]
\end{defn}

The reduced $C^{*}$-algebra of $\GG$ is obtained from the algebra $C_{c}(\GG)$ by completing it with respect to the norm obtained from the canonical family of $*$-representations $\pi_{x}:C_{c}(\GG)\rightarrow\LL(L^{2}(\GG_{x}))$, $x\in\GG^{(0)}$, called the regular representation.  For a groupoid dynamical system $(\AF,\GG,\alpha)$ the situation is slightly more complicated - namely, we must complete the convolution algebra $\Gamma_{c}(\GG;r^{*}\AF)$ with respect to the norm obtained from a particular family \emph{Hilbert modules} constructed from $\AF$. 

\begin{prop}\label{rednorm}
	Let $(\AF,\GG,\alpha)$ be a groupoid dynamical system.  Then for each $x\in\GG^{(0)}$, the completion $L^{2}(\GG_{x};r^{*}\AF)$ of $\Gamma_{c}(\GG_{x};r^{*}\AF)$ in the $\AF_{x}$-valued inner product
	\[
	\langle\xi,\eta\rangle_{x}:=(\xi^{*}*\eta)(x) = \int_{\GG}\alpha_{u}\big(\xi(u^{-1})^{*}\eta(u^{-1})\big)\,d\lambda^{x}(u),
	\]
	defined for $\xi,\eta\in\Gamma_{c}(\GG_{x};r^{*}\AF)$, is a Hilbert $\AF_{x}$-module, with right $\AF_{x}$-action given by
	\[
	(\xi\cdot a)(u):=\xi(u)\alpha_{u}(a),\hspace{7mm}u\in\GG_{x}
	\]
	for all $\xi\in\Gamma_{c}(\GG_{x};r^{*}\AF)$ and $a\in\AF_{x}$.
	
	Moreover, for each $x\in\GG^{(0)}$, a representation $\pi_{x}:\Gamma_{c}(\GG;r^{*}\AF)\rightarrow\LL\big(L^{2}(\GG_{x};r^{*}\AF)\big)$ is defined by the formula
	\[
	\pi_{x}(f)\xi(u):=f*\xi(u) = \int_{\GG}f(v)\alpha_{v}\big(\xi(v^{-1}u)\big)\,d\lambda^{r(u)}(v),\hspace{7mm}u\in\GG_{x}
	\]
	for $f\in\Gamma_{c}(\GG;r^{*}\AF)$ and $\xi\in\Gamma_{c}(\GG_{x};r^{*}\AF)$.  The family $\{\pi_{x}\}_{x\in\GG^{(0)}}$ is called the \textbf{regular representation} associated to $(\AF,\GG,\alpha)$.
\end{prop}

\begin{proof}
	It is clear from inspection that for $x\in\GG^{(0)}$, $\xi,\eta\in\Gamma_{c}(\GG_{x};r^{*}\AF)$ and $a\in\AF_{x}$, the formulae defining $\langle\xi,\eta\rangle_{x}$ and $(\xi\cdot a)$ make sense.  Moreover, we have
	\[
	\langle\xi,\eta\rangle^{*}_{x} = (\xi^{*}*\eta)^{*}(x) = (\eta^{*}*\xi)(x) = \langle\eta,\xi\rangle_{x}
	\]
	since $*$ is an involution, and
	\begin{align*}
		\langle\xi,\eta\cdot a\rangle_{x} =& \int_{\GG}\alpha_{u}\big(\xi(u^{-1})^{*}(\eta\cdot a)(u^{-1})\big)\,d\lambda^{x}(u)\\ =& \int_{\GG}\alpha_{u}\big(\xi(u^{-1})^{*}\eta(u^{-1})\alpha_{u^{-1}}(a)\big)\,d\lambda^{x}(u) = \langle\xi,\eta\rangle_{x}a
	\end{align*}
	since $\alpha_{u^{-1}} = \alpha_{u}^{-1}$ as $C^{*}$-isomorphisms.  Consequently \cite[p. 4]{Lance}, the completion $L^{2}(\GG_{x};r^{*}\AF)$ is indeed a Hilbert $\AF_{x}$-module.
	
	The extension of $\pi_{x}$ to a representation $\pi_{x}:\Gamma_{c}(\GG;r^{*}\AF)\rightarrow\LL\big(L^{2}(\GG_{x};r^{*}\AF)\big)$ follows from the argument of Khoshkam and Skandalis \cite[3.6]{koshskand}.  Specifically, Khoshkam and Skandalis show that for $\tilde{f}\in\Gamma_{c}(\GG;s^{*}\AF)$ (see Remark \ref{convention}) and $\tilde{\xi}\in\Gamma_{c}(\GG_{x};\AF_{x})$, the formula
	\[
	\tilde{\pi}_{x}(\tilde{f})\xi(u):= \int_{\GG}\alpha_{w}\big(\tilde{f}(uw)\big)\tilde{\xi}(w^{-1})\,d\lambda^{s(u)}(w),\hspace{7mm}u\in\GG_{x}
	\]
	extends to a representation $\tilde{\pi}_{x}:\Gamma_{c}(\GG;s^{*}\AF)\rightarrow\LL\big(L^{2}(\GG_{x};\AF_{x})\big)$, where $L^{2}(\GG_{x};\AF_{x})$ is the Hilbert $\AF_{x}$-module that is the completion of $\Gamma_{c}(\GG_{x};\AF_{x})$ in the $\AF_{x}$-valued inner product $\langle\tilde{\xi},\tilde{\eta}\rangle_{x}:=\int_{\GG}\tilde{\xi}(u^{-1})^{*}\tilde{\eta}(u^{-1})\,d\lambda^{x}(u)$.  Observe that the map $U_{x}:\Gamma_{c}(\GG_{x};\AF_{x})\rightarrow\Gamma_{c}(\GG_{x};r^{*}\AF)$ defined by
	\[
	\big(U_{x}\tilde{\xi}\big)(u):=\alpha_{u}\big(\tilde{\xi}(u)\big)
	\]
	satisfies $\big(U_{x}(\tilde{\xi}\cdot a)\big)(u) = \alpha_{u}(\tilde{\xi}(u))\alpha_{u}(a) = \big(U_{x}\tilde{\xi}\big)\cdot a(u)$ for all $a\in\AF_{x}$, $\tilde{\xi}\in\Gamma_{c}(\GG_{x};\AF_{x})$ and $u\in\GG_{x}$, and satisfies
	\begin{align*}
		\langle U_{x}\tilde{\xi},\eta\rangle_{x} =& \int_{\GG}\alpha_{u}\big((U_{x}\tilde{\xi})(u^{-1})^{*}\eta(u^{-1})\big)\,d\lambda^{x}(u) = \int_{\GG}\tilde{\xi}(u^{-1})^{*}\alpha_{u}\big(\eta(u^{-1})\big)\,d\lambda^{x}(u)\\ =& \int_{\GG}\tilde{\xi}(u^{-1})^{*}\big(U^{-1}_{x}\eta\big)(u^{-1})\,d\lambda^{x}(u)  = \langle\tilde{\xi},U_{x}^{-1}\eta\rangle_{x}
	\end{align*}
	for all $\tilde{\xi}\in\Gamma_{c}(\GG_{x};\AF_{x})$ and $\eta\in\Gamma_{c}(\GG_{x};r^{*}\AF)$.  Thus it extends to a unitary isomorphism $U_{x}:L^{2}(\GG_{x};\AF_{x})\rightarrow L^{2}(\GG_{x};r^{*}\AF)$ of Hilbert $\AF_{x}$-modules.  Now we observe that
	\begin{align*}
		\big(U_{x}\circ\tilde{\pi}_{x}(\tilde{f})\circ U_{x}^{*}\big)\xi(u) =& \alpha_{u}\bigg(\big(\tilde{\pi}_{x}(\tilde{f})\circ U_{x}^{*}\big)\xi(u)\bigg)\\ =& \alpha_{u}\bigg(\int_{\GG}\alpha_{w}\big(\tilde{f}(uw)\big)\big(U_{x}^{*}\xi\big)(w^{-1})\,d\lambda^{x}(w)\bigg)\\ =& \alpha_{u}\bigg(\int_{\GG}\alpha_{w}\big(\tilde{f}(uw)\big)\alpha_{w}\big(\xi(w^{-1})\big)\,d\lambda^{x}(w)\bigg)\\ =& \int_{\GG}\alpha_{uw}\big(\tilde{f}(uw)\big)\alpha_{uw}\big(\xi(w^{-1})\big)\,d\lambda^{x}(w)\\ =& \int_{\GG}\alpha_{v}\big(\tilde{f}(v)\big)\alpha_{v}\big(\xi(v^{-1}u)\big)\,d\lambda^{r(u)}(v) = \pi_{x}\big(\alpha\circ \tilde{f}\big)\xi(u)
	\end{align*}
	for any $\xi\in\Gamma_{c}(\GG_{x};r^{*}\AF)$, so that $\pi_{x}$ is unitarily equivalent to the representation $\tilde{\pi}_{x}$ and is consequently itself a representation by \cite[Section 3.6]{koshskand}.
\end{proof}

Completing in the norm obtained from the Hilbert module representations of Proposition \ref{rednorm} we obtain the reduced crossed product algebra.

\begin{defn}\label{redcrossed}
	The completion $A\rtimes_{r}\GG$ of $\Gamma_{c}(\GG;r^{*}\AF)$ in the norm
	\[
	\|f\|_{A\rtimes_{r}\GG}:=\sup_{x\in\GG^{(0)}}\|\pi_{x}(f)\|_{L^{2}(\GG_{x};r^{*}\AF)}
	\]
	is a $C^{*}$-algebra called the \textbf{reduced crossed product algebra} associated to the dynamical system $(\AF,\GG,\alpha)$.
\end{defn}

\begin{rmk}\normalfont
	The proof of Proposition \ref{rednorm} together with \cite[Section 3.7]{koshskand} implies that the $C^{*}$-algebra defined in Definition \ref{redcrossed} coincides with that given in \cite{koshskand}.
\end{rmk}

\begin{rmk}\normalfont
	It is important to note that, for any $\GG$-algebra $A$, the calculations of \cite[Section 3.6]{koshskand} show that $\|\cdot\|_{A\rtimes_{r}\GG}\leq \|\cdot\|_{I}$.  Consequently, we have that $\|\cdot\|_{A\rtimes_{r}\GG}\leq\|\cdot\|_{A\rtimes\GG}$ (cf. Definition \ref{fullcrossed}).
\end{rmk}

\begin{rmk}\normalfont\label{inductivelimit}
	In doing calculations with crossed product $C^{*}$-algebras, one usually must work with the algebra $\Gamma_{c}(\GG;r^{*}\AF)$ associated to the upper-semicontinuous $C^{*}$-bundle $\AF$.  In working with this algebra, it is frequently necessary to utilise a finer topology than the (reduced or full) norm topology.  If $p_{\BF}:\BF\rightarrow \GG^{(0)}$ is an upper-semicontinuous Banach bundle, one says that a net $z_{\lambda}$ in $\Gamma_{c}(\GG;r^{*}\BF)$ \emph{converges in the inductive limit topology to $z\in\Gamma_{c}(\GG;r^{*}\BF)$} if $z_{\lambda}\rightarrow z$ uniformly, and if there is a compact set $C$ in $\GG$ such that $z$ and all of the $z_{\lambda}$ vanish off $C$ \cite[Remark 3.9]{muhwil}.  It is then clear that if $z_{\lambda}$ converges to $z$ in the inductive limit topology, then $z_{\lambda}$ \emph{also} converges to $z$ in the $I$-norm topology.  Consequently $z_{\lambda}\rightarrow z$ in both the full and reduced $C^{*}$-norm topologies.
\end{rmk}

\section{The descent map}

We now prove that Kasparov's descent map \cite{KasparovEqvar} exists and continues to function as expected in the equivariant setting for locally Hausdorff groupoids.  While the relevant definitions in the Hausdorff case \cite{legall} are relatively simple using balanced tensor products, in the locally Hausdorff case one must use the bundle picture, which requires slightly more work.

Suppose that we are given a $\GG$-algebra $(B = \Gamma_{0}(\GG;r^{*}\BF),\beta)$ and a $\GG$-Hilbert $B$-module $(\EE = \Gamma_{0}(\GG^{(0)};r^{*}\EF),W)$. Observe that $\Gamma_{c}(\GG;r^{*}\EF)$ admits the $\Gamma_{c}(\GG;r^{*}\BF)$-valued inner product $\langle\cdot,\cdot\rangle_{\GG}$ defined for $\xi,\xi'\in\Gamma_{c}(\GG;r^{*}\EF)$
\[
\langle\xi,\xi'\rangle_{\GG}(u):=\int_{\GG}\beta_{v}\big(\langle\xi(v^{-1}),\xi'(v^{-1}u)\rangle_{s(v)}\big)\,d\lambda^{r(u)}(v),\hspace{7mm}u\in\GG,
\]
which is positive by \cite[Proposition 6.8]{muhwil} and the argument given in \cite[p. 116]{williamscrossed}.  The space $\Gamma_{c}(\GG;r^{*}\EF)$ also admits a right action of $\Gamma_{c}(\GG;r^{*}\BF)$ given by
\[
(\xi\cdot f)(u):=\int_{\GG}\xi(v)\cdot\beta_{v}\big(f(v^{-1}u)\big)\,d\lambda^{r(u)}(v),\hspace{7mm}u\in\GG
\]
for $\xi\in\Gamma_{c}(\GG;r^{*}\EF)$ and $f\in\Gamma_{c}(\GG;r^{*}\BF)$, with $\cdot$ denoting the (fibrewise) right action of $\BF$ on $\EF$.  By completing in the norm attained from either the reduced or full crossed product we obtain \cite[p. 4]{Lance} a Hilbert module.

\begin{defn}
	Given a $\GG$-algebra $(B = \Gamma_{0}(\GG^{(0)};\BF),\beta)$ and a $\GG$-Hilbert $B$-module $(\EE = \Gamma_{0}(\GG^{(0)};r^{*}\EF),W)$, we define $\EE\rtimes\GG$ and $\EE\rtimes_{r}\GG$ to be the completions of $\Gamma_{c}(\GG;r^{*}\EF)$ in the norms
	\[
	\|\xi\|_{\EE\rtimes\GG}:=\|\langle\xi,\xi\rangle_{\GG}\|_{B\rtimes\GG}^{\frac{1}{2}},\hspace{7mm}\|\xi\|_{\EE\rtimes_{r}\GG}:=\|\langle\xi,\xi\rangle_{\GG}\|_{B\rtimes_{r}\GG}^{\frac{1}{2}}
	\]
	respectively.  The space $\EE\rtimes\GG$ is a Hilbert $B\rtimes\GG$-module, while the space $\EE\rtimes_{r}\GG$ is a Hilbert $B\rtimes_{r}\GG$-module, which we will refer to respectively as the \textbf{full and reduced crossed products} of $\EE$ by $\GG$.
\end{defn}

\begin{rmk}\label{fullred}\normalfont
	If $B$ is a $\GG$-algebra and $\EE$ a Hilbert $B$-module, then since the reduced $C^{*}$-norm on $\Gamma_{c}(\GG;r^{*}\BF)$ is bounded by the full norm, we see that the reduced norm on $\Gamma_{c}(\GG;r^{*}\EF)$ is bounded by the full norm also.
\end{rmk}

We pull back operators on equivariant modules to their crossed products as follows.

\begin{prop}\label{adj}
	Let $(B = \Gamma_{0}(\GG^{(0)};\BF),\beta)$ be a $\GG$-algebra and let $(\EE = \Gamma_{0}(\GG^{(0)};\EF),W)$ be a $\GG$-Hilbert $B$-module.  Given $T\in\LL(\EE) = \Gamma_{b}(X;\LL(\EF))$, the operator $r^{*}(T)$ defined on $\Gamma_{c}(\GG;r^{*}\EF)$ defined by
	\[
	\big(r^{*}(T)\,\xi\big)(u):=T(r(u))\xi(u)
	\]
	extends to an operator $r^{*}(T)\in\LL(\EE\rtimes_{r}\GG)$, and consequently to an operator $r^{*}(T)\in\LL(\EE\rtimes\GG)$ also.
\end{prop}

\begin{proof}
	That $r^{*}(T)$ is $\Gamma_{c}(\GG;r^{*}\BF)$-linear is clear by the $B$-linearity of $T$.  Moreover, for any $\xi\in\Gamma_{c}(\GG;r^{*}\EF)$ we have
	\begin{align*}
		\langle r^{*}(T)\,\xi,r^{*}(T)\,\xi\rangle_{\GG}(u) =& \int_{\GG}\beta_{v}\big(\langle T(s(v))\xi(v^{-1}),T(s(v))\xi(v^{-1}u)\rangle_{s(v)}\big)\,d\lambda^{r(u)}(v)\\ \leq&\int_{\GG}\|T(s(v))\|^{2}_{\EF_{s(v)}}\beta_{v}\big(\langle\xi(v^{-1}),\xi(v^{-1}u)\rangle_{s(v)}\big)\,d\lambda^{r(u)}(v)\\ \leq& \sup_{x\in\GG^{(0)}}\|T(x)\|^{2}_{\EF_{x}}\int_{\GG}\beta_{v}\big(\langle\xi(v^{-1}),\xi(v^{-1}u)\big)\,d\lambda^{r(u)}(v),
	\end{align*}
	from which we deduce that $\|r^{*}(T)\,\xi\|_{\EE\rtimes_{r}\GG}\leq\|T\|_{\LL(\EE)}\|\xi\|_{\EE\rtimes_{r}\GG}$.  Consequently, $r^{*}(T)$ extends to a bounded, $B\rtimes_{r}\GG$-linear operator on all of $\EE\rtimes_{r}\GG$.  Remark \ref{fullred} then tells us that $r^{*}(T)$ also extends to a bounded $B\rtimes\GG$-linear operator on $\EE\rtimes\GG$.  It is then easily checked in both cases that $r^{*}(T)$ is adjointable, with adjoint $r^{*}(T^{*})$.
	\end{proof}

\begin{rmk}\normalfont
	The identifications of Remark \ref{rmkpullbax} show that when $\GG$ is Hausdorff, our $r^{*}(T)$ defined on $\EE\rtimes_{r}\GG$ and $\EE\rtimes\GG$ agrees with $T\hotimes1$ defined on $\EE\hotimes_{B}(B\rtimes_{r}\GG)$ and $\EE\hotimes_{B}(B\rtimes\GG)$ respectively.
\end{rmk}

\begin{rmk}\normalfont\label{unbddadj}
	If, in the notation of Proposition \ref{adj}, one has an \emph{unbounded} $B$-linear operator $T:\dom(T)\rightarrow\EE$, one defines $\dom(r^{*}(T)):=\Gamma_{c}(\GG;r^{*}\mathfrak{dom}(T))$, where $\mathfrak{dom}(T)$ is exactly as in Definition \ref{operatorpullback}, and obtains $r^{*}(T):\dom(r^{*}(T))\rightarrow \EE\rtimes_{r}\GG$ and $r^{*}(T):\dom(r^{*}(T))\rightarrow\EE\rtimes\GG$ using the same formula as in Proposition \ref{adj}.
\end{rmk}

That the crossed product of an equivariant $A$-$B$-correspondence is an $A\rtimes_{r}\GG$-$B\rtimes_{r}\GG$-correspondence turns out to require some mildly nontrivial argumentation, which does not currently appear in the literature.  We give a full proof below.

\begin{prop}
	Let $(A = \Gamma_{0}(\GG^{(0)};\AF),\alpha)$ and $(B = \Gamma_{0}(\GG^{(0)};\BF),\beta)$ be $\GG$-algebras and let $(\EE = \Gamma_{0}(\GG^{(0)};\EF),W)$ be a $\GG$-Hilbert module.  Then if $\pi:A\rightarrow\LL(\EE)$ is an equivariant representation, the formula
	\[
	(f\cdot\xi)(u):=\int_{\GG}\pi_{r(v)}(f(v))W_{v}\big(\xi(v^{-1}u)\big)\,d\lambda^{r(u)}(v),\hspace{7mm}u\in\GG
	\]
	defined for $f\in\Gamma_{c}(\GG;r^{*}\AF)$ and $\xi\in\Gamma_{c}(\GG;r^{*}\EF)$ determines representations $\pi\rtimes_{r}\GG:A\rtimes_{r}\GG\rightarrow\LL(\EE\rtimes_{r}\GG)$ and $\pi\rtimes\GG:A\rtimes\GG\rightarrow\LL(\EE\rtimes\GG)$.
\end{prop}

\begin{proof}
	First observe that $\pi\rtimes_{r}\GG$ is $*$-preserving in the sense that for any $f\in\Gamma_{c}(\GG;r^{*}\AF)$, and $\xi,\xi'\in\Gamma_{c}(\GG;r^{*}\EF)$, the element $\langle f\cdot\xi,\xi'\rangle_{\GG}(u)$ of $\BF_{r(u)}$ is equal to
	\begin{align*}
		&\int_{\GG}\int_{\GG}\beta_{v}\big(\langle\pi_{s(v)}(f(w))W_{w}\big(\xi(w^{-1}v^{-1})\big),\xi'(v^{-1}u)\rangle_{s(v)}\big)\,d\lambda^{s(v)}(w)\,d\lambda^{r(u)}(v)\\ =& \int_{\GG}\int_{\GG}\beta_{v}\big(\langle W_{w}\big(\xi(w^{-1}v^{-1})\big),\pi_{s(v)}(f(w)^{*})\xi'(v^{-1}u)\rangle_{s(v)}\big)\,d\lambda^{s(v)}(w)\,d\lambda^{r(u)}(v)\\ =& \int_{\GG}\int_{\GG}\beta_{vw}\big(\langle\xi(w^{-1}v^{-1}),\varepsilon_{w^{-1}}\big(\pi_{s(v)}(f(w)^{*})\big)W_{w^{-1}}\big(\xi'(v^{-1}u)\big)\rangle_{s(w)}\big)\,d\lambda^{s(v)}(w)\,d\lambda^{r(u)}(v)\\ =& \int_{\GG}\int_{\GG}\beta_{\overline{v}}\big(\langle\xi(\overline{v}^{-1}),\pi_{s(\overline{v})}(f^{*}(\overline{w}))W_{\overline{w}}\big(\xi'(\overline{w}^{-1}\overline{v}^{-1}u)\big)\rangle_{s(\overline{v})}\big)\,d\lambda^{s(\overline{v})}(\overline{w})\,d\lambda^{r(u)}(\overline{v})\\ =& \langle\xi,(f^{*}\cdot\xi')\rangle_{\GG}(u)
	\end{align*}
	for all $u\in\GG$.  Here we have used the equivariance of the representation $\pi$ and the substitutions $\overline{v}:=vw$ and $\overline{w}:=w^{-1}$ in going from the third line to the fourth.
	
	To prove that $f\mapsto f\cdot$ extends to a homomorphism $A\rtimes_{r}\GG\rightarrow\LL(\EE\rtimes_{r}\GG)$ we must show that $\|f\cdot\xi\|\leq\|f\|_{A\rtimes_{r}\GG}\|\xi\|_{\EE\rtimes_{r}\GG}$ for all $f\in\Gamma_{c}(\GG;r^{*}\AF)$ and $\xi\in\Gamma_{c}(\GG;r^{*}\EF)$.  Because $\EF$, $\AF$ and $\BF$ are in general \emph{different} bundles, we cannot use the techniques of \cite[Theorem 1.4]{ouchi} or its analogue in the non Hausdorff case \cite{tu}; nor can we use the techniques of \cite[Section 8]{muhwil} since we are working with reduced crossed products and not full crossed products.  Observe, however, that if $p:\EF\rightarrow\overline{\pi(\AF)\EF}$ denotes the family of projections
	\[
	p_{x}:\EF_{x}\rightarrow\overline{\pi_{x}(\AF_{x})\EF_{x}},\hspace{7mm}x\in\GG^{(0)},
	\]
	then $\pi_{x}(a)e = \pi_{x}(a)(p_{x}e)$ for all $x\in\GG^{(0)}$, $a\in\AF_{x}$ and $e\in\EF_{x}$.  Then for any $f\in\Gamma_{c}(\GG;r^{*}\AF)$ and $\xi\in\Gamma_{c}(\GG;r^{*}\EF)$ we compute
	\begin{align*}
		(f\cdot\xi)(u) =& \int_{\GG}\pi_{r(v)}(f(v))W_{v}\big(\xi(v^{-1}u)\big)\,d\lambda^{r(u)}(v)\\ =& \int_{\GG}W_{v}\bigg(\varepsilon_{v^{-1}}\big(\pi_{r(v)}(f(v))\big)\big(\xi(v^{-1}u)\big)\bigg)\,d\lambda^{r(u)}(v).
	\end{align*}
	Using the equivariance of $\pi$, we then see that
	\begin{align*}
(f\cdot\xi)(u)  =& \int_{\GG}W_{v}\big(\pi_{s(v)}(\alpha_{v^{-1}}(f(v)))\xi(v^{-1}u)\big)\,d\lambda^{r(u)}(v)\\ =& \int_{\GG}W_{v}\big(\pi_{s(v)}(\alpha_{v^{-1}}(f(v)))p_{s(v)}\xi(v^{-1}u)\big)d\lambda^{r(u)}(v)\\ =& \int_{\GG}\pi_{r(v)}(f(v))W_{v}\big(p_{s(v)}\xi(v^{-1}u)\big)\,d\lambda^{r(u)}(v)\\ =& (f\cdot(r^{*}(p)\xi))(u),
	\end{align*}
	so we can assume without loss of generality that $\xi\in\Gamma_{c}(\GG;r^{*}(p\EF))$.  Since the representation $\pi$ of $\AF$ on $p\EF$ is (fibrewise) nondegenerate, by \cite[Proposition 6.8]{muhwil} there exists a sequence $(e_{i})_{i\in\NB}$ in $\Gamma_{c}(\GG;r^{*}\AF)$ such that for any $\xi\in\Gamma_{c}(\GG;r^{*}(p\EF))$, the sequence $(\pi(e_{i})\xi)_{i\in\NB}$ converges to $\xi$ in the inductive limit topology on $\Gamma_{c}(\GG;r^{*}\EF)$.  Thus we can assume without loss of generality that $\xi$ is of the form $g\cdot\xi'$ for $g\in\Gamma_{c}(\GG;r^{*}\AF)$ and $\xi'\in\Gamma_{c}(\GG;r^{*}(p\EF))$.  For such $\xi$ we estimate
	\begin{align*}
		\langle f\cdot\xi,f\cdot\xi\rangle_{\GG} =& \langle\xi',(g^{*}*f^{*}*f*g)\cdot\xi'\rangle_{\GG}\\ \leq& \|f\|^{2}_{A\rtimes_{r}\GG}\langle\xi',(g^{*}*g)\cdot\xi'\rangle_{\GG}\\ =&\|f\|^{2}_{A\rtimes_{r}\GG}\langle\xi,\xi\rangle_{\GG}.
	\end{align*}
	Hence
	\[
	\|f\cdot\xi\|_{\EE\rtimes_{r}\GG}\leq\|f\|_{A\rtimes_{r}\GG}\|\xi\|_{\EE\rtimes_{r}\GG},
	\]
	so $f\mapsto f\cdot$ does indeed define a homomorphism $\pi\rtimes_{r}\GG:A\rtimes_{r}\GG\rightarrow\LL(\EE\rtimes_{r}\GG)$.  By Remark \ref{fullred}, $f\mapsto f\cdot$ also extends to a homomorphism $\pi\rtimes\GG:A\rtimes\GG\rightarrow\LL(\EE\rtimes\GG)$.
\end{proof}

Observe now that if $(\EE,F)$ is a $\GG$-equivariant Kasparov $A$-$B$-module, then we can form the pairs $(\EE\rtimes\GG,r^{*}(F))$ and $(\EE\rtimes_{r}\GG,r^{*}(F))$.  It is in this way that we obtain the following theorem.

\begin{thm}\label{descent1}
	Let $A$ and $B$ be $\GG$-algebras.  For any $\GG$-equivariant Kasparov $A$-$B$-module $(\EE,F)$, we have $(\EE\rtimes\GG,r^{*}(F))\in\EB(A\rtimes\GG,B\rtimes\GG)$ and $(\EE\rtimes_{r}\GG,r^{*}(F))\in\EB(A\rtimes_{r}\GG,B\rtimes_{r}\GG)$ (see Remark \ref{nonequi}). The induced maps
	\[
	j^{\GG}:KK^{\GG}(A,B)\rightarrow KK(A\rtimes\GG,B\rtimes\GG),\hspace{5mm}j^{\GG}_{r}:\KK^{\GG}(A,B)\rightarrow KK(A\rtimes_{r}\GG,B\rtimes_{r}\GG)
	\]
	are homomorphisms of abelian groups, and are compatible with the Kasparov product in the sense that if $x\in KK^{\GG}(A,C)$ and $y\in KK^{\GG}(C,B)$, then we have
	\[
	j^{\GG}(x\otimes_{C}y) = j^{\GG}(x)\otimes_{C\rtimes\GG}j^{\GG}(y),\hspace{7mm} j^{\GG}_{r}(x\otimes y) = j^{\GG}_{r}(x)\otimes_{C\rtimes_{r}\GG}j^{\GG}_{r}(y).
	\]
\end{thm}

\begin{proof}
	That $\EE\rtimes\GG$ and $\EE\rtimes_{r}\GG$ are countably generated can be seen by taking a countable approximate identity $\{f_{i}\}_{i\in\NB}$ for $C_{c}(\GG)$ in the inductive limit topology and a countable generating set $\{e_{j}\}_{j\in\NB}$ for $\Gamma_{0}(X;\EF)$.  Then, letting $\cdot$ denote pointwise multiplication, the countable collection $\{f_{i}\cdot r^{*}e_{j}\}_{i,j\in\NB}$ is a generating set for both $\EE\rtimes\GG$ and $\EE\rtimes_{r}\GG$ (note that since each $f_{i}$ is a finite sum of functions that are continuous and supported in Hausdorff open subsets, so too is each $f_{i}\cdot r^{*}e_{j}$).  We have already proved in Proposition \ref{adj} that the operator $r^{*}(F)$ makes sense in both $\LL(\EE\rtimes\GG)$ and $\LL(\EE\rtimes_{r}\GG)$.  The result is otherwise proved in the same manner as in \cite[Theorem 3.11]{KasparovEqvar}, once one replaces Kasparov's $C_{c}(G,A)$, $C_{c}(G,\mathscr{E})$ and $C_{c}(G,\mathscr{K}(\mathscr{E}))$ with our $\Gamma_{c}(\GG;r^{*}\AF)$, $\Gamma_{c}(\GG;r^{*}\EF)$ and $\Gamma_{c}(\GG;r^{*}\KK(\EF))$ respectively.
\end{proof}

We end the paper by noting that the descent map can also be applied to an unbounded equivariant Kasparov module, to yield an unbounded Kasparov module for the associated crossed product algebras.

\begin{prop}\label{descent}
	Let $A$ and $B$ be $\GG$-algebras, 
	and let $(\AA,\EE,D)$ be a $\GG$-equivariant unbounded 
	Kasparov $A$-$B$-module.  Let $\AS\subset\AF$ be as in Definition \ref{unbdd}.  Then
	\[
	(\Gamma_{c}(\GG;r^{*}\AS),\EE\rtimes\GG,\overline{r^{*}(D)}),\hspace{7mm}(\Gamma_{c}(\GG;r^{*}\AS),\EE\rtimes_{r}\GG,\overline{r^{*}(D)})
	\] 
	are, respectively, an unbounded Kasparov $A\rtimes\GG$-$B\rtimes\GG$-module and an unbounded Kasparov $A\rtimes_{r}\GG$-$B\rtimes_{r}\GG$-module.
\end{prop}

\begin{proof}
	Here $\overline{r^{*}(D)}$ means the closure of the operator $r^{*}(D)$ defined in Remark \ref{unbddadj}.  The same arguments used in \cite[Proposition 2.11]{macr1} now apply, provided one switches out the half-densities therein for a choice of Haar system on $\GG$.
\end{proof}

	\bibliographystyle{amsplain}
	\bibliography{references}

\providecommand{\bysame}{\leavevmode\hbox to3em{\hrulefill}\thinspace}
\providecommand{\MR}{\relax\ifhmode\unskip\space\fi MR }
\providecommand{\MRhref}[2]{%
  \href{http://www.ams.org/mathscinet-getitem?mr=#1}{#2}
}
\providecommand{\href}[2]{#2}
\begin{thebibliography}{10}

\bibitem{iakovos}
I.~Androulidakis and G.~Skandalis, \emph{{A Baum-Connes conjecture for singular
  foliations}}, arXiv:1509.05862, 2015.

\bibitem{baajjulg1}
S.~Baaj and P.~Julg, \emph{{Theorie bivariante de Kasparov et operateurs
  non-born\'{e}s dans les $C^{*}$-modules Hilbertiens}}, {C. R. Acad. Sci.
  Paris} \textbf{296} (1983), 875--878.

\bibitem{bmv}
S.~Brain, B.~Mesland, and W.~D. van Suijlekom, \emph{{Gauge Theory for Spectral
  Triples and the Unbounded Kasparov Product}}, {J. Noncommut. Geom.}
  \textbf{10} (2016), 135--206.

\bibitem{cgrs1}
A.~Carey, V.~Gayral, A.~Rennie, and F.~Sukochev, \emph{{Integration on locally
  compact noncommutative spaces}}, {J. Funct. Anal.} \textbf{263} (2012),
  383--414.

\bibitem{cgrs2}
\bysame, \emph{{Index theory for locally compact noncommutative geometries
  (Memoirs of the American Mathematical Society)}}, {Amer. Math. Soc.}, 2014.

\bibitem{cprs1}
A.~Carey, J.~Phillips, A.~Rennie, and F.~Sukochev, \emph{{The Hochschild class
  of the Chern character for semifinite spectral triples}}, {J. Funct. Anal.}
  \textbf{213} (2004), 111--153.

\bibitem{cprs2}
\bysame, \emph{{The local index formula in semifinite von Neumann algebras I:
  Spectral flow}}, {Adv. Math.} \textbf{202} (2006), 451--516.

\bibitem{cprs3}
\bysame, \emph{{The local index formula in semifinite von Neumann algebras II:
  The even case}}, {Adv. Math.} \textbf{202} (2006), 517--554.

\bibitem{cprs4}
\bysame, \emph{{The Chern character of semifinite spectral triples}}, {J.
  Noncommut. Geom.} \textbf{2} (2008), 141--193.

\bibitem{folops}
A.~Connes, \emph{{A survey of foliations and operator algebras}}, Proc. Sympos.
  Pure, 1982, pp.~38--521.

\bibitem{CM}
A.~Connes and M.~Moscovici, \emph{{The local index formula in noncommutative
  geometry}}, {GAFA} \textbf{5} (1995), 174 -- 243.

\bibitem{cs}
A.~Connes and G.~Skandalis, \emph{The longitudinal index theorem for
  foliations}, Publ. RIMS \textbf{20} (1984), 1139--1183.

\bibitem{felldoran}
J.~M.~G. Fell and R.~S. Doran, \emph{{Representations of $*$-Algebras, Locally
  Compact Groups, and Banach $*$-Algebraic Bundles}}, vol.~1, {Academic Press},
  1988.

\bibitem{higsonkastech}
N.~Higson, \emph{{On a Technical Theorem of Kasparov}}, {J. Funct. Anal.}
  \textbf{73} (1987), 102--112.

\bibitem{higsonlocind}
\bysame, \emph{The local index formula in noncommutative geometry},
  {Contemporary Developments in Algebraic K-Theory, ICTP Lecture Notes},
  vol.~15, 2003, pp.~444--536.

\bibitem{hofmann}
K.~H. Hofmann, \emph{{Bundles and sheaves are equivalent in the category of
  Banach spaces}}, {$K$-theory and Operator Algebras} (B.~B. Morel and I.~M.
  Singer, eds.), vol. 575, Springer, 1977.

\bibitem{kl13}
J.~Kaad and M.~Lesch, \emph{{Spectral flow and the unbounded Kasparov
  product}}, {Adv. Math.} \textbf{248} (2013), 495--530.

\bibitem{KasparovTech}
G.~G. Kasparov, \emph{{The operator $K$-functor and extensions of
  $C^*$-algebras}}, Math. USSR Izv. \textbf{16} (1981), 513--572.

\bibitem{KasparovEqvar}
\bysame, \emph{{Equivariant KK-theory and the Novikov conjecture}}, {Invent.
  Math.} \textbf{91} (1988), 147--202.

\bibitem{koshskand2}
M.~Khoshkam and G.~Skandalis, \emph{{Regular representations of groupoid
  $C^{*}$-algebras and applications to inverse semigroups}}, {J. reine angew.
  Math.} \textbf{{546}} (2002), {47--72}.

\bibitem{koshskand}
\bysame, \emph{{Crossed products of $C^{*}$-algebras by groupoids and inverse
  semigroups}}, {J. Operat. Theor.} \textbf{51} (2004), 255--279.

\bibitem{kucer}
D.~Kucerovsky, \emph{{The $KK$-product of unbounded modules}}, {$K$-theory}
  \textbf{11} (1997), {17--34}.

\bibitem{Lance}
E.~C. Lance, \emph{{Hilbert C*-Modules: A Toolkit for Operator Algebraists
  (London Mathematical Society Lecture Note Series)}}, Cambridge University
  Press, 1995.

\bibitem{legall}
P.-Y. Le~Gall, \emph{{Th{\'e}orie de Kasparov {\'e}quivariante et
  groupo{\"\i}des}}, Ph.D. thesis, {Universit\'{e} Paris Diderot Paris 7},
  1994.

\bibitem{legall2}
\bysame, \emph{{Th\'{e}orie de Kasparov \'{e}quivariante et groupo\:{i}des.
  I.}}, {$K$-Theory} \textbf{16} (1999), 361--390.

\bibitem{LRV}
S.~Lord, A.~Rennie, and J.~C. V\'{a}rilly, \emph{{Riemannian manifolds in
  noncommutative geometry}}, {J. Geom. Phys.} \textbf{62} (2012), 1611--1638.

\bibitem{macr1}
L.~MacDonald and A.~Rennie, \emph{{The Godbillon-Vey invariant in equivariant
  $KK$-theory}}, arXiv:1811.04603, 2018.

\bibitem{ouchi}
M.~Macho-Stadler and M.~O'Uchi, \emph{{Correspondence of groupoid
  $C^{*}$-algebras}}, {J. Operat. Theor.} \textbf{42} ({1999}), {103--119}.

\bibitem{mes14}
B.~Mesland, \emph{{Unbounded bivariant $K$-theory and correspondences in
  noncommutative geometry}}, {J. Reine Angew. Math.} \textbf{691} (2014),
  101--172.

\bibitem{mr15}
B.~Mesland and A.~Rennie, \emph{{Nonunital spectral triples and metric
  completenesss in unbounded $KK$-theory}}, {J. Funct. Anal.} \textbf{271}
  (2016), 2460--2538.

\bibitem{muhwil}
P.~S. Muhly and D.~Williams, \emph{{Renault's equivalence theorem for groupoid
  crossed products}}, {New York J. Math. Mono.} \textbf{3} (2008), 1--87.

\bibitem{pat}
A.~Paterson, \emph{Groupoids, inverse semigroups, and their operator algebras
  (progress in mathematics)}, Birkh{\"a}user, 2012.

\bibitem{pierrot}
F.~Pierrot, \emph{{Bimodules de Kasparov non born{\'e}s {\'e}quivariants pour
  les groupo{\"\i}des topologiques localement compacts}}, {C.R. Acad. Sci.
  Paris} \textbf{342} (2006), 661--663.

\bibitem{seda2}
A.~K. Seda, \emph{{On the continuity of Haar measure on topological
  groupoids}}, {Proc. Amer. Math. Soc.} \textbf{{96}} ({1986}), {115--120}.

\bibitem{sk1}
G.~Skandalis, \emph{{Some remarks on Kasparov theory}}, {J. Funct. Anal.}
  \textbf{56} (1984), no.~3, 337 -- 347.

\bibitem{tu}
J.-L. Tu, \emph{{Non-Hausdorff groupoids, proper actions and $K$-theory}},
  {Doc. Math.} \textbf{9} (2004), 565--597.

\bibitem{twistedKstacks}
J.~L. Tu, P.~Xu, and C.~Laurent-Gengoux, \emph{{Twisted $K$-theory of
  differentiable stacks}}, {Annales scientifiques de l'\'{E}cole Normale
  Sup\'{e}rieure} \textbf{37} (2004), 841--910.

\bibitem{williamscrossed}
D.P. Williams, \emph{{Crossed Products of C*-algebras}}, {Mathematical Surveys
  and Monographs}, {Amer. Math. Soc.}, 2007.

\end{thebibliography}
	
\end{document}